\documentclass[11pt]{amsart}
\usepackage{amssymb,amsthm,amsmath}
\RequirePackage[dvipsnames,usenames]{xcolor}
\usepackage{hyperref}
\usepackage{mathtools,etoolbox}
\usepackage[all]{xy}
\usepackage{tikz}
\usepackage{enumitem}
\usepackage[english]{babel}
\usepackage{graphics}
\usepackage{color}

\hypersetup{
bookmarks,
bookmarksdepth=3,
bookmarksopen,
bookmarksnumbered,
pdfstartview=FitH,
colorlinks,backref,hyperindex,
linkcolor=Sepia,
anchorcolor=BurntOrange,
citecolor=MidnightBlue,
citecolor=OliveGreen,
filecolor=BlueViolet,
menucolor=Yellow,
urlcolor=OliveGreen
}

\newcommand{\factor}[2]{\left. \raise 2pt\hbox{\ensuremath{#1}} \right/
        \hskip -2pt\raise -2pt\hbox{\ensuremath{#2}}}

\makeatletter
\renewcommand\subsection{
  \renewcommand{\sfdefault}{pag}
  \@startsection{subsection}%
  {2}{0pt}{.8\baselineskip}{.4\baselineskip}{\raggedright
    \sffamily\itshape\small\bfseries
  }}
\renewcommand\section{
  \renewcommand{\sfdefault}{phv}
  \@startsection{section} %
  {1}{0pt}{\baselineskip}{.8\baselineskip}{\centering
    \sffamily
    \scshape
    \bfseries
}}
\makeatother

\usepackage[left=1.02in,top=1.0in,right=1.02in,bottom=1.0in]{geometry}
\usepackage{multirow}

\setlist[enumerate]{leftmargin=0.8cm}
\setlist[itemize]{leftmargin=0.8cm}
\setlist[description]{leftmargin=0.0cm}

\newtheorem{thm}{Theorem}[section]

\newtheorem{lemma}[thm]{Lemma}

\newtheorem{que}{Question}
\newtheorem{conj}[que]{Conjecture}

\newtheorem{theorem}[thm]{Theorem}
\newtheorem{corollary}[thm]{Corollary}
\newtheorem{proposition}[thm]{Proposition}

\theoremstyle{definition}
\newtheorem{definition}[thm]{Definition}
\newtheorem{example}[thm]{Example}
\theoremstyle{remark}
\newtheorem{remark}[thm]{Remark}

\theoremstyle{remark}

\usepackage{color}

\DeclareMathOperator{\Supp}{Supp}

\DeclareMathOperator{\Exc}{Exc}

\DeclareMathOperator{\mult}{mult}
\DeclarePairedDelimiter{\Set}{\lbrace}{\rbrace} 
\def\bydef{\coloneqq}

\numberwithin{equation}{section}

\newcommand{\PP}{{\mathbb P}}
\newcommand{\bP}{{\mathbb P}}
\newcommand{\C}{{\mathbb C}}

\newcommand{\Q}{{\mathbb Q}}
\newcommand{\R}{{\mathbb R}}
\newcommand{\Z}{{\mathbb Z}}

\newcommand{\Qbar}{{\overline{\Q}}}

\newcommand{\kbar}{{\overline{k}}}

\newcommand{\calC}{{\mathcal C}}

\newcommand{\calL}{{\mathcal L}}

\newcommand{\calO}{{\mathcal O}}

\newcommand{\calX}{{\mathcal X}}

\newcommand{\calZ}{{\mathcal Z}}

\usepackage[mathscr]{euscript}

\DeclareMathOperator{\supp}{supp}

\DeclareMathOperator{\ord}{ord}
\DeclareMathOperator{\Sym}{Sym}

\DeclareMathOperator{\Pic}{Pic}

\DeclareMathOperator{\pr}{pr}

\numberwithin{equation}{section}
\numberwithin{table}{section}

\title{Nonspecial varieties and Generalized Lang-Vojta conjectures}
\author{Erwan Rousseau}
\address{Erwan Rousseau \newline
	Institut Universitaire de France \& Aix Marseille Univ\\ 
	CNRS, Centrale Marseille, I2M, Marseille, France 
	\& Freiburg Institute for Advanced Studies, University of Freiburg,
Albertstr.19, 79104 Freiburg, Germany} 
\email{erwan.rousseau@univ-amu.fr}
\author{Amos Turchet}
\address{Amos Turchet \newline
Dipartimento di Matematica e Fisica, Universit\'a degli studi Roma 3, L.go S. L. Murialdo 1, 00146 Roma \\
previously: Centro di Ricerca Matematica Ennio De Giorgi, Scuola Normale Superiore, Palazzo Puteano, Piazza dei cavalieri, 3, 56100 Pisa, Italy}
\email{amos.turchet@uniroma3.it}
\author{Julie Tzu-Yueh Wang}
\address{Julie Tzu-Yueh Wang \newline Institute of Mathematics, Academia Sinica 
No.\ 1, Sec.\ 4, Roosevelt Road
Taipei 10617, Taiwan}
\email{jwang@math.sinica.edu.tw}


\begin{document}
\begin{abstract}
  We construct a family of fibered threefolds $X_m \to (S , \Delta)$ such that $X_m$ has no \'etale cover that dominates a variety of general type but it dominates the orbifold $(S,\Delta)$ of general type. Following Campana, the threefolds $X_m$ are called \emph{weakly special} but not \emph{special}. The Weak Specialness Conjecture predicts that a weakly special variety defined over a number field has a potentially dense set of rational points. We prove that if $m$ is big enough the threefolds $X_m$ present behaviours that contradict the function field and analytic analogue of the Weak Specialness Conjecture. We prove our results by adapting the recent method of Ru and Vojta. We also formulate some generalizations of known conjectures on exceptional loci that fit into Campana's program and prove some cases over function fields.
\end{abstract}
\maketitle

\section{Introduction}
A fundamental problem in Diophantine Geometry is to describe the distribution of rational points in an algebraic variety $X$ defined over a number field $k$. The expectation is that the global geometry of $X$ controls the distribution of $X(k)$. In the case in which $X$ is a variety of general type, conjectures of Lang and Vojta (see e.g. \cite[Conjecture 3.7]{Lang91}, \cite[Conjecture 3.4.3]{Vojta87}) predict the existence of a proper Zariski closed subset, called the \emph{exceptional set}, such that in its complement rational points $X(L)$ should be finite for every finite extension $L \supset k$; moreover the exceptional set is expected to be independent of the field of definition. On the opposite side there are the varieties whose rational points are potentially dense, i.e. there exists a finite field extension $L/k$ such that the set of $L$-rational points $X(L)$ is Zariski dense. The conjectures of Lang and Vojta thus predict that rational points in a variety of general type are not potentially dense.\smallskip

It is natural to look for geometrical properties characterizing varieties where the set of rational points is potentially dense. The conjectures mentioned above imply that for such a property to be satisfied on $X$, not only $X$ but none of its finite \'etale covers should dominate a positive-dimensional variety of general type: such varieties are called \emph{weakly special} (see Section \ref{wspecial} for details). In \cite[Conjecture 1.2]{HT} the following Conjecture was stated, which we will denote as the \emph{Weak Specialness Conjecture}: a variety $X$ has a potentially dense set of rational points if and only if $X$ is weakly special.\smallskip

In \cite{Ca04}, Campana introduced the stronger notion of \emph{specialness}: a variety $X$ is \emph{special} if it does not admit any fibration of general type, in the sense of Campana (see Section \ref{special} for details). Campana conjectured that specialness rather than weak-specialness should characterize potential density. Note that these two characterizations differ from each other (even if they agree in dimension up to two) since there exist projective varieties which are weakly special but not special. Such examples were first constructed in \cite{BT} as simply-connected threefolds equipped with an elliptic fibration $\pi: X \to S$ where $S$ has Kodaira dimension $1$. Nevertheless, these examples do not explicitly contradict the Weak Specialness Conjecture since we still lack a method to control the distribution of rational points on these varieties. Therefore it is natural to consider the analogous problem in the analytic and function fields settings. The goal of this article is to study the two conjectures in these settings giving evidence towards Campana conjecture.\smallskip

In the analytic setting, the Green--Griffiths--Lang conjecture predicts that entire curves in varieties of general type should be contained in a proper Zariski closed subset, the already mentioned exceptional set. Following this analogy, Campana has conjectured that specialness (and therefore potential density) should correspond to the existence of Zariski dense entire curves (see \cite{Ca04}). The analytic analogue of the Weak Specialness Conjecture would imply that a weakly special variety admits a Zariski dense entire curve. This was already disproved in \cite{CP}, where it was proven that in some of the examples of weakly special but not special varieties constructed in \cite{BT}, all entire curves are algebraically degenerate (i.e. with non Zariski dense image).\smallskip

Similarly, the conjectures of Lang and Vojta admit analogues for varieties defined over function fields. In our setting, Vojta's height conjecture predicts that given a general type variety $X$ defined over a characteristic zero function field $\kappa(\calC)$ of a smooth integral curve $\calC$, and an ample line bundle $\calL$, there exists a positive constant $\alpha$ such that sections $s: \calC \to X$ not contained in the exceptional set, satisfy $\deg s^*\calL \leq \alpha (2 g(\calC) - 2)$. We say that varieties $X$ satisfying this condition are \emph{pseudo algebraically hyperbolic} and we refer to Definition \ref{def:alghyp} for further details and references. A function field analogue of the Weak Specialness Conjecture would predict that a weakly special variety $X$ does not dominate a positive-dimensional pseudo-algebraically hyperbolic variety.

One of the goals of this paper is to present a list of new examples of varieties that are weakly special but not special: these are of the form $\pi: X_m \to (S,\Delta_\pi)$ where $X_m$ is a threefold, $\pi$ is an elliptic fibration, the compactification $\overline{S}$ of $S$ is a blow-up of $\PP^2$ and $\Delta_\pi = (1-1/m)\widetilde{D_1}$ is an orbifold divisor (see Section \ref{EC} for the precise definition).

The pairs $(S,\Delta_\pi)$ can be seen as orbifold generalizations of surfaces appearing in \cite{CoZa10} which provide examples of simply-connected quasi-projective surface with a non Zariski dense set of integral points. For this class of examples we show that the function field and analytic Weak Specialness Conjecture fail. In particular, we prove the following:

\begin{theorem}[see Proposition \ref{prop:blow_up} and Theorem \ref{deg}]
 In the settings above, there exists $m_0$ such that for all $m \geq m_0$, the following hold:
 \begin{enumerate}
   \item the base of the fibration $X_m \to (S,\Delta_\pi)$ is pseudo algebrically hyperbolic;
   \item every entire curve in $X_m$ is algebraically degenerate.
 \end{enumerate}
   \label{thm:main}
\end{theorem}

As an example we obtain the following explicit corollary. 
\begin{corollary}
Let $L_1, L_2, L_3$ and $L_4$ be four lines in general position in $\PP^2$ and let $\overline{S}$ be the blow up of $\PP^2$ in three points $P_1,P_2,P_3$ where $P_i \in L_i$ for $i=1,2,3$. If we denote by $\tilde{L}_i$ the strict transform of the line $L_i$ and by $S = \overline{S} \setminus \tilde{L}_2 + \tilde{L}_3 + \tilde{L}_4$, the surface $S \setminus \tilde{L}_1$ is simply connected, and the orbifold $(\overline{S}, \Delta_m = (1 - 1/m)\tilde{L}_1 + \tilde{L}_2 + \tilde{L}_3 + \tilde{L}_4)$ is of general type for every $m \geq 2$, i.e. the divisor $K_{\overline{S}} + \Delta_m$ is big. Then, $X_m = (T \times_{\PP^1} \overline{S}) \to (\overline{S}, \Delta_m)$ where $T$ is an elliptic fibration (see \ref{ex:1fib} for the construction), is a simply connected threefold that does not dominate any surface or curve of general type, and therefore it is weakly special (see Theorem \ref{tbt}).
Then, the following holds:
\begin{enumerate}
	\item a degree bound for maps of the form $\pi \circ s: \calC \to (\overline{S},\Delta_m)$;
	\item algebraic degeneracy of entire curves in $X_m$.
\end{enumerate}
\end{corollary}

In \cite[Problem 3.7]{HassettT} Hassett and Tschinkel ask whether a pair that does not admit any \'etale cover that dominates a variety of log general type, i.e. it is weakly special, admit a potentially dense set of integral points. Theorem \ref{thm:main} answers the analogous question over function fields in the negative.

To prove Theorem \ref{thm:main} we combine Corvaja-Zannier's degeneracy statements in \cite{CZAnnals} with the recent framework of Ru and Vojta \cite{RV19} to obtain a generalization of \cite[Main Theorem]{CZAnnals} in the orbifold setting for the function field and analytic case.

\begin{theorem}[see Corollary \ref{cor:CZ_orb} and Theorem \ref{th:CZ_nev}]
  \label{th:CZ_nev_intro}
		 Let $X\subset \PP^m$ be a complex nonsingular projective surface and $D = D_1 + \dots + D_q$ be a divisor with $q \geq 2$, such that
		 \begin{enumerate}
			 \item No three of the components $D_i$ meet at a point;
			 \item There exists a choice of positive integers $p_i$ such that 
				 \begin{itemize}
					 \item the divisor $D_p := p_1 D_1 + p_2 D_2 + \dots + p_q D_q$ is ample;
					 \item The following inequality holds:
						 \[
							 2 D_p^2 \xi_i > (D_p \cdot D_i) \xi_i^2 + 3 D_p^2 p_i,
						 \]
						 for every $i = 1,\dots,q$ where $\xi_i$ is the minimal positive solution of the equation $D_i^2 x^2 - 2(D_p \cdot D_i) x + D_p^2 = 0$.
				 \end{itemize}
		 \end{enumerate}
		Let $\Delta$ be the $\Q$-divisor defined as
	\[
		\Delta = \sum_{i=1}^q \left( 1 - \frac{1}{m_i} \right) D_i.
	\]
	Then, there exists a positive integer $m$ such that, if $m_i \geq m$ for every $i$, 
	\begin{enumerate}
		\item every orbifold entire curve $\psi: \C \to (X,\Delta)$, is algebraically degenerate;
		\item $(X,\Delta)$ is pseudo algebraically hyperbolic.
	\end{enumerate}
	 \end{theorem}

The key point in the proof of Theorem \ref{th:CZ_nev_intro} is an analytic and a function field version of Ru-Vojta's Theorem \cite[General Theorem]{RV19} with truncation, which we develop  in Theorem \ref{trungeneral} and Theorem \ref{thm:ffconst} by adding a ramification term to Ru-Vojta's Theorem. Even though this type of generalization is far from reach in the arithmetic setting, our approach via the Ru-Vojta's Theorem gives a new interpretation of some important work such as \cite{CZAnnals}, \cite{Levin}, and \cite{Aut2} that has contributed fundamental ideas and techniques in the development of the Ru-Vojta's Theorem. We expect that this point of view and the truncated version of Ru-Vojta's Theorem will have further applications.

	 \subsection*{Structure of the paper} In Section \ref{special} we recall basic facts about special varieties and orbifolds, following Campana. Then, in Section \ref{conj}, we present a general framework where we generalize Lang's notion of exceptional loci to nonspecial varieties and Campana orbifolds. In Section \ref{wspecial} we discuss a general procedure to construct weakly special but not special varieties, generalizing \cite{BT}. In Section \ref{sec:vojta} we prove Theorem \ref{th:CZ_nev_intro} and we apply it in Section \ref{EC} to prove Theorem \ref{thm:main}. Finally, in Section \ref{sec:SMT} we prove some generalizations of Ru and Vojta \cite[General Theorem]{RV19} for function fields and entire curves.

\subsection*{Acknowledgements}
We thank Laura Capuano, Pietro Corvaja, Lionel Darondeau, Carlo Gasbarri, Ariyan Javanpeykar, Stefan Kebekus and Min Ru for useful discussions. We thank Min Ru and Paul Vojta for providing us with an earlier copy of their paper. We are especially grateful to Fr\'ed\'eric Campana for all the discussions on this subject and especially his explanations on the construction of weakly special but not special varieties. We thank the referees for their comments and suggestions which greatly improved the paper.

This collaboration was initiated at the conference ``Topics On Nevanlinna Theory And Complex Hyperbolicities'' at the Shanghai Center for Mathematical Sciences; we thank the organizers for making this possible. Part of this work was finalized during the visit of the first and second named author to the Freiburg Institute for Advanced Studies; they thank the Institute for providing an excellent working environment.

\section{Special varieties}\label{special}
We collect here basic definitions and constructions related to special varieties, while referring to \cite{Ca04} for more details.

\subsection{Special Manifolds via Bogomolov sheaves}
Let \(X\) be a connected complex nonsingular projective variety of complex dimension \(n\). 
For a  rank-one coherent subsheaf \(\calL\subset\Omega^p_X\), denote by \(H^0(X,\calL^{m})\) the space of sections of \(\Sym^m(\Omega^p_X)\) which take values in \(\calL^{m}\) (where as usual $\calL^m \bydef \calL^{\otimes m}$).
The \emph{Iitaka dimension} of $\calL$ is \(\kappa(X,\calL)\bydef\max_{m>0}\{\dim(\Phi_{\calL^m}(X))\}\), i.e. the maximum dimension of the image of rational maps
\(\Phi_{\calL^{m}}\colon X\dasharrow\PP(H^0(X,\calL^{m}))\) defined at the generic point of \(X\), 
where by convention \(\dim(\Phi_{\calL^{m}}(X))\bydef-\infty\) if there are no global sections.
Thus \(\kappa(X,\calL)\in \Set{-\infty, 0,1,\dotsc, \dim(X)}\). 
In this setting, a theorem of Bogomolov in \cite{Bog} shows that, if \(\calL\subset\Omega_{X}^{p}\), then \(\kappa(X,\calL)\leq p\).
\begin{definition}
  Let $p >0$. A rank one saturated coherent sheaf \(\calL\subset\Omega_{X}^{p}\)
  is called a \emph{Bogomolov sheaf} if \(\kappa(X,\calL)=p\), i.e. if $\calL$ has the largest possible Iitaka dimension.
\end{definition}

The following remark shows that the presence of Bogomolov sheaves on $X$ is related to the existence of fibrations $f: X\to Y$ where $Y$ is of general type.

\begin{remark}
  If \(f\colon X\to Y\) is a fibration (by which we mean a surjective morphism with connected fibers) and \(Y\) is a variety of general type of dimension \(p>0\), then the saturation of \(f^*(K_Y)\) in \(\Omega^p_X\) is a Bogomolov sheaf of \(X\), 
\end{remark}

Campana introduced the notion of specialness in \cite[Definition 2.1]{Ca04} to generalize the absence of fibration as above.

\begin{definition}\label{def:special}
  A nonsingular variety \(X\) is said to be \textsl{special} (or \textsl{of special type}) if there is no Bogomolov sheaf on $X$. 
  A projective variety is said to be \textsl{special} if some (or any) of its resolutions are special.
\end{definition}

By the previous remark if there is a fibration $X \to Y$ with $Y$ a positive dimensional variety of general type then $X$ is nonspecial. In particular, if \(X\) is of general type of positive dimension, $X$ is not special.

\subsection{Special Manifolds via orbifold bases} 
Campana gave a characterization of special varieties using his theory of orbifolds. We briefly recall the construction.

Let $Z$ be a normal connected compact complex variety.
An \textsl{orbifold divisor} \(\Delta\) is a linear combination \(\Delta\bydef\sum_{\Set{D\subset Z}}c_{\Delta}(D)\cdot D\), 
where \(D\) ranges over all prime divisors of \(Z\),
the \textsl{orbifold coefficients} are rational numbers \(c_{\Delta}(D)\in[0,1]\cap\Q\) such that all but finitely many are zero.
Equivalently,
\[
  \Delta
  =
  \sum_{\{D\subset Z\}}
  \left(
    1-\frac{1}{m_{\Delta}(D)}
  \right)
  \cdot
  D,
\]
where only finitely many \textsl{orbifold multiplicities} \(m_{\Delta}(D)\in\Q_{\geq 1}\cup \{+\infty\}\) are larger than \(1\).

An orbifold pair is a pair \((Z,\Delta)\) where $\Delta$ is an orbifold divisor; they interpolate between the compact case where \(\Delta=\varnothing\) and the pair \((Z,\varnothing)=Z\) has no orbifold structure, and the {\it open}, or {\it purely-logarithmic case} where $c_j = 1$ for all $j$, and we identify \((Z,\Delta)$ with $Z\setminus\Supp(\Delta)\).

When \(Z\) is smooth and the support \(\Supp(\Delta)\bydef\cup D_j\) of \(\Delta\) has normal crossings singularities, we say that \((Z,\Delta)\) is {\it smooth}. When all multiplicities \(m_j\) are integral or \(+\infty\), we say that the orbifold pair \((Z,\Delta)\) is {\it integral}, and when every $m_j$ is finite it may be thought of as a virtual ramified cover of \(Z\) ramifying at order \(m_j\) over each of the \(D_j\)'s.\medskip

Consider a fibration \(f\colon X\to Z\) between normal connected complex projective varieties. In general, the geometric invariants (such as \(\pi_1(X), \kappa(X),\dots\)) of \(X\) do not coincide with the `sum' of those of the base (\(Z\)) and of the generic fiber (\(X_\eta\)) of \(f\).
Replacing \(Z\) by the `orbifold base' \((Z,\Delta_f)\) of \(f\), which encodes the multiple fibers of \(f\), leads in some favorable cases to such an additivity (on suitable birational models at least).

\begin{definition}[Orbifold base of a fibration]
  \label{dob} 
  Let \(f\colon X \to Z\) be a fibration as above and let $\Delta$ be an orbifold divisor on \(X\). 
  We shall define the \textsl{orbifold base} \((Z,\Delta_{f})\) of \((f,\Delta)\) as follows: to each irreducible Weil divisor \(D\subset Z\) we assign the multiplicity \(m_{(f,\Delta)}(D)\bydef\inf_k\{t_k\cdot m_\Delta(F_k)\}\), where the scheme-theoretic fiber of $D$ is \(f^*(D)=\sum_k t_k.F_k+R\), \(R\) is an \(f\)-exceptional divisor of \(X\) with \(f(R)\subsetneq D\) and \(F_k\) are the irreducible divisors of \(X\) which map surjectively to \(D\) via $f$.
\end{definition}
\begin{remark}
  Note that the integers \(t_k\) are well-defined, even if \(X\) is only assumed to be normal.
\end{remark}

Let \((Z,\Delta)\) be an orbifold pair. Assume that \(K_Z+\Delta\) is \(\Q\)-Cartier (this is the case if \((Z,\Delta)\) is smooth, for example): we will call it the \emph{canonical bundle} of \((Z,\Delta)\). Similarly we will denote by the {\it canonical dimension} of $(Z,\Delta)$ the Kodaira dimension of $K_Z + \Delta$ i.e. $\kappa(Z,\Delta): = \kappa(Z,\calO_Z(K_Z+\Delta))$. Finally, we say that the orbifold \((Z,\Delta)\) is of {\it general type} if \(\kappa(Z,\Delta)=\dim(Z)\).

\begin{definition}
A fibration \(f\colon X\to Z\) is said to be \textsl{of general type} if $Z$ is positive dimensional and the orbifold base \((Z,\Delta_{f})\) is of general type.
\end{definition}

The idea of Campana was that in order to characterize varieties that have a potentially dense set of rational points it was not sufficient to exclude the presence of \'etale covers that dominate varieties of general type. One would need to exclude the presence of every fibration of general type in the above sense. This turns out to be equivalent to the specialness condition of Definition \ref{def:special} as proven by Campana.

\begin{theorem}[\protect{see \cite[Theorem 2.27]{Ca04}}]
  A variety \(X\) is special if and only if it has no fibrations of general type.
\end{theorem}

\subsection{Orbifold Morphisms}
We recall here the main definition of orbifold morphism, following \cite[Definition 2.3]{Ca11}, in the special case in which the source is a curve.

\begin{definition}
  Let $(X,\Delta_X)$ and $(\calC,\Delta_\calC)$ two orbifold pairs, with $X$ and $\calC$ normal projective varieties defined over a field $k$, $\calC$ a curve, and $\Delta_X,\Delta_\calC$ two orbifold divisors of the form 
		\[
			\Delta_X = \sum_{i=1}^r \left( 1 - \frac{1}{m_i}\right) \Delta_i \qquad \Delta_\calC = \sum_{j=1}^s \left( 1 - \frac{1}{n_j}\right) P_j
		\]
		For every prime divisor $D$ of $X$ we denote by $m_X(D)$ its multiplicity, i.e. the number $m_i$ if $D = \Delta_i$ or 0 otherwise. For every point $Q \in \calC$ we similarly define $n_\calC(Q)$.
		A morphism $f: \calC \to X$ is an \emph{orbifold morphism}, denoted by $f: (\calC, \Delta_\calC) \to (X, \Delta_X)$ if 
		\begin{enumerate}
			\item $f(\calC)$ is \emph{not} contained in $\Delta_X$;
			\item for every prime divisor $D$ of $X$, if $f^*D = \sum_k n_k Q_k$, one has $n_k \cdot n_\calC(Q_k) \geq m_X(D)$.
		\end{enumerate}
\end{definition}

We recall the following definition introduced by Demailly in \cite{De97} in the compact case, while the logarithmic and orbifold analogue were introduced in \cite{Chen} and \cite{Rou} respectively.

\begin{definition}
	Let $(X,D)$ be a pair of a non-singular projective variety $X$ defined over $k$ and let $D$ be a normal crossing divisor on $X$. We say that $(X,D)$ is \emph{algebraically hyperbolic} if there exists a ample line bundle $\calL$ on $X$ and a positive constant $\alpha$ such that for every non-singular projective curve $\calC$ and every morphism $\varphi: \calC \to X$ the following holds:
	\begin{equation}
		\deg \varphi^* \calL \leq \alpha \cdot (2g(\calC) - 2 + N_\varphi^{[1]}(D)).
		\label{eq:alghyp}
	\end{equation}
	where $N_\varphi^{[1]}(D)$ is the cardinality of the \emph{support} of $\varphi^*(D)$.
	We say that $(X,D)$ is \emph{pseudo} algebraically hyperbolic if there exists a proper closed subvariety $Z$ of $X$ such that \eqref{eq:alghyp} holds for every morphism $\varphi: \calC \to X$ such that $\varphi(\calC) \nsubseteq Z$.
	\label{def:alghyp}
\end{definition}

The notion of \emph{pseudo}-algebraic hyperbolicity was defined first in \cite{Jav19, Jav19b}. In Lang's terminology \cite{Lan86}, the notion is the ``pseudofication'' of the notion of algebraic hyperbolicity. 
We note that when Equation \eqref{eq:alghyp} holds for an ample line bundle $\calL$ then it holds for \emph{every} ample line bundle with possibly a different constant $\alpha$.

\begin{remark}
	Note that the degree $\deg \varphi^* \calL$, is a Weil Height for $\varphi$ viewed as a point in the function field $k(\calC)$, with respect to the ample line bundle $\calL$. In the next sections, we refer to it both as a degree and as the height $h_\calL(\varphi)$.
\end{remark}

Definition \ref{def:alghyp} can be extended to the orbifold setting as follows.

\begin{definition}\label{def:ah_orb}
	Let $(X,\Delta)$ be an orbifold; we say that $(X,\Delta)$ is \emph{algebraically hyperbolic} if there exists a ample line bundle $\calL$ on $X$ and a positive constant $\alpha$ such that for every non-singular projective curve $\calC$ and every \emph{orbifold} morphism $\psi: (\calC,\Delta_C) \to (X,\Delta)$ the following holds:
	\begin{equation}
		\deg \psi^* \calL \leq \alpha \cdot (2g(\calC) - 2 + \deg \Delta_C).
		\label{eq:algorb}
	\end{equation}
	We say that $(X,\Delta)$ is \emph{pseudo} algebraically hyperbolic if there exists a proper closed subvariety $Z$ of $X$ such that \eqref{eq:algorb} holds for every orbifold morphism $\psi: (\calC, \Delta_C) \to (X,\Delta)$ such that $\psi(\calC) \nsubseteq Z$.
\end{definition}

\section{Generalized Lang conjectures}\label{conj}
In this section we propose a generalization of the conjectures of Lang and Vojta compatible with Campana's dicothomy between special and nonspecial.

\subsection{Arithmetic and Analytic Exceptional sets}
In \cite{Lan86}, Lang introduced the following exceptional sets.
\begin{definition}
Let $X$ be a projective variety defined over a field $k$.
\begin{enumerate}
  \item If $k$ is a finitely generated field of characteristic zero, then the \emph{Diophantine exceptional set} $\Exc_{\text{dio}}(X)$ is the smallest Zariski-closed subset $Z$ of $X$ such that $(X\setminus Z)(L)$ is finite for all finite extensions $L \supset k$.
  \item If $k=\C$ then the \emph{holomorphic exceptional set} $\Exc_{\text{hol}}(X)$ is the Zariski closure of the union of all entire curves i.e. images of non-constant holomorphic maps $f: \C \to X$.
\end{enumerate}
\label{def:exceptional}
\end{definition}

Lang has conjectured \cite{Lan86} that $X$ is of general type if and only if both these exceptional sets are proper subsets of $X$. Moreover he conjectured that $\Exc_{\text{dio}}(X) = \Exc_{\text{hol}}(X)$.

Given Campana's notion of special variety, it is natural to try to extend the notion of exceptional sets to \emph{nonspecial} varieties in the sense of Campana \cite{Ca04}. 

The starting point is the following conjecture, formulated by Campana \cite[Conjecture 9.2 and Conjecture 9.20]{Ca04}.
\begin{conj}\label{conj:Camp}
Let $X$ be a projective variety defined over a field $k$.
\begin{enumerate}
\item If $k$ is a number field, then $X(L)$ is not Zariski dense for all finite extensions $L \supset k$ if and only if $X$ is not special
\item If $k=\C$ then $f(\C)$ is not Zariski dense for all entire curves $f: \C \to X$ if and only if $X$ is not special.
\end{enumerate}
\end{conj}

When one considers nonspecial varieties, it is easy to see that there are examples of nonspecial varieties where the two exceptional sets in Definition \ref{def:exceptional} are the entire variety. As an example, consider the nonspecial variety $X=\calC\times \bP^1$ defined over a number field $k$, where $\calC$ is a smooth projective curve of genus $\geq 2$: in this case for every rational point $P$ in $\calC(L)$ the curve $\{P\} \times \PP^1$ is in $\Exc_{\text{dio}}(X)$. Therefore since $X(\kbar)$ is dense in $X$, Lang's Diophantine exceptional set coincides with $X$. Similarly one can show that $\Exc_{\text{hol}}(X) = X$.
We suggest that, in order to define meaningful exceptional sets, one should consider the projectivized tangent bundle $\PP(T_X)$. 

\begin{definition}\leavevmode
  \label{def:exc}
\begin{enumerate}
\item If $k$ is a number field, the Diophantine exceptional set $\Exc^1_{\text{dio}}(X)$ is the smallest Zariski-closed subset $Z$ of $\PP(T_X)$ such that for or all finite extensions $L \supset k$, $\bP(T_{Y_L}) \subset Z$, where $Y_L$ is the Zariski closure of $X(L)$.

\item If $k=\C$, the holomorphic exceptional set $\Exc^1_{\text{hol}}(X)$ is the Zariski closure of the union of all entire curves $g: \C \to \PP(T_X)$ obtained as liftings of entire curves $f:\C \to X$.
\end{enumerate}
\end{definition}

The main motivation behind Definition \ref{def:exc} is Campana's Core construction in \cite[Section 3]{Ca04}. Given a smooth projective variety $X$ there is a functorial fibration $c_X: X \to C(X)$, called the \emph{core} of $X$ such that the fibers of $c_X$ are special varieties and the base $C(X)$ is either a point or an orbifold of general type. The idea behind considering $\PP(T_X)$ as the natural space where the exceptional set live, is that the core of $X$ identifies the ``special direction'' in $\PP(T_X)$ and therefore, assuming that $C(X)$ has positive dimension, i.e. $X$ is nonspecial, this exceptional set should not be the whole $\PP(T_X)$.

Therefore we propose the following generalization of Lang conjecture for nonspecial varieties.

\begin{conj}\label{conj:LO}
Let $X$ be a projective variety defined over a field $k$.
\begin{enumerate}
\item If $k$ is a number field then $\Exc^1_{\text{dio}}(X) \neq \PP(T_X)$ if and only if $X$ is not special
\item If $k=\C$ then $\Exc^1_{\text{hol}}(X) \neq \PP(T_X)$ if and only if $X$ is not special.
\end{enumerate}
\end{conj}

In the example above $X = \calC \times \PP^1$ with $\calC$ an hyperbolic curve, we see that for every number field $L$ the closed subvariety $Y_L$, the closure of $X(L)$, is the union of finitely many rational curves, corresponding to the fibers of $\pr_1 = c_X$ over the $L$-rational points of $\calC$. Thus $\Exc^1_{\text{dio}}(X) = \PP\pr_2^*(T_{\PP^1})$ and in particular $\Exc^1_{\text{dio}}(X) \neq \PP(T_X)$.

If $X$ is a closed subvariety of an abelian variety, then conjecture \ref{conj:LO} holds  by Faltings's proof of Mordell-Lang  \cite{FaltingsLang} and the work of Ueno, Bloch-Ochiai-Kawamata \cite{Kawamata} and Yamanoi  \cite{Yamanoi} on closed subvarieties of abelian varieties.

Comparing Conjecture \ref{conj:Camp} and Conjecture \ref{conj:LO} suggests interesting questions. In particular a strong uniform degeneracy statement in $\PP(T_X)$ should imply a non density statement in $X$.
\begin{conj}
Let $X$ be a projective variety defined over a field $k$.
\begin{enumerate}
	\item If $k$ is a number field and $\Exc^1_{\text{dio}}(X) = \PP(T_X)$ then $X(k)$ is potentially dense.
	\item If $k=\C$ and $\Exc^1_{\text{hol}}(X) \neq \PP(T_X)$ then entire curves $f:\C \to X$ are algebraically degenerate i.e. the images $f(\C)$ are not Zariski dense.
\end{enumerate}
\end{conj}

Interestingly, some examples of this phenomenon have already been proved in the foliated setting. If one considers a complex projective manifold $X$ equipped with a holomorphic (singular) foliation $\mathcal{F}$ such that all entire curves $f:\C \to X$ are tangent to this foliation, then it is proved in \cite{McQ98} that if $X$ is a surface of general type and $\mathcal{F}$ is a foliation by curves, then entire curves are indeed algebraically degenerate. More recently, the same conclusion is shown in \cite{LRT} if $\mathcal{F}$ is a transversely hyperbolic foliation of codimension $1$ and $X$ an arbitrary complex projective manifold $X$ (not necessarily of general type!).

\subsection{Function Fields}
Given the well-known analogy between number fields and function fields, we formulate the above conjectures in the function field setting. Let $\kappa$ be an algebraically closed field of characteristic 0, let $\calC$ be a smooth complex projective curve and let $\kappa(\calC)$ be its function field over $\kappa$. Let $X$ be a proper variety defined over $\kappa(\calC)$ and let $g: \calX \to \calC$ be the fibration associated to a proper model of $X$ over $\calC$. In this setting, $\kappa(\calC)$-rational points correspond to sections $s: \calC \to \calX$ of $g$ and points defined over a finite extension $\kappa(\calC') \supset \kappa(\calC)$ correspond to sections of the base change $\calX' := \calX \times_\calC \calC' \to \calC'$ via a cover $\calC' \to \calC$. We say that $X$ is \emph{isotrivial} if there exists a model $\calX$ over $\calC$ and a cover $\calC' \to \calC$ such that $\calX'$ is birational to $F\times \calC'$ where $F$ is a manifold.

In this setting, Lang's conjecture (in its weak version) can be formulated in the following way (see \cite[Historical appendix: algebraic families]{Lan86} for a historical discussion).
\begin{conj}\label{conj:LFF}
  Let $X$ be non-isotrivial. If $X$ is of general type then for all finite extensions $\kappa(\calC') \supset \kappa(\calC)$, $\kappa(\calC')$-points in $X$ are not Zariski dense.
\end{conj}

Special cases of Conjecture \ref{conj:LFF} have been proved in the literature. The analogue of Mordell's conjecture over function fields was proved by Manin \cite{Manin63} and Grauert \cite{Gra65}. A higher dimensional version is obtained in \cite{Nog81} in the case of ample cotangent bundle.

Similarly, the generalization of Conjecture \ref{conj:Camp} is formulated as follows.
\begin{conj}
	\label{conj:ff_nonsp}
	Let $X$ be non-isotrivial. Then $X$ is not special if and only if for all finite extensions $\kappa(\calC') \supset \kappa(\calC)$, $\kappa(\calC')$-points in $X$ are not Zariski dense.
\end{conj}

Lang's stronger conjecture predicts that the existence of an \emph{exceptional set}, and moreover such set should be independent of the field of definition of the points. 

\begin{remark}
If one wants to formulate a strong version of Conjecture \ref{conj:LFF} analogous to Lang Conjecture for varieties of general type over number fields, non-isotriviality is not enough as shown by the example of a product of an isotrivial variety with a non-isotrivial one. One has to add some assumption on the non-density of the subvarieties which are dominated by isotrivial varieties. Such a strong conjecture is considered in \cite{Ca11} Conjecture 13.21.
\end{remark}

Like in the number field case, a stronger form of Conjecture \ref{conj:LFF} cannot be translated directly to nonspecial varieties. Instead we propose the following analogue of Conjecture \ref{conj:LO} in the function field setting.
\begin{conj}
Let $X$ be non-isotrivial. $X$ is not special if and only if for all finite extensions $\kappa(\calC') \supset \kappa(\calC)$ there is a proper algebraic subset $Z \subset \PP(T_X)$ such that  sections $s: \calC' \to \calX'$ whose liftings $s_1: \calC'\to \PP(T_\calX)$ are not contained in $Z$ are finite.
\end{conj}

In the number field case, conjectures of Vojta (see \cite[Conjecture 3.4.3]{Vojta87} and \cite[Conjecture 24.3]{Vojta11}) predict a height bound for rational points in varieties of general type outside of the exceptional locus, which implies a non-density statement. 

In the function field setting, a height bound is also expected, which is the content of the Lang-Vojta conjecture.

\begin{conj}\label{conj:LVFF}
If $X$ is of general type and $\calL$ is an ample line bundle on $X$, then there exist a proper algebraic subset $Z \subsetneq X$ and a positive constant $\alpha$ such that sections $s: \calC \to \calX$ not contained in $\calZ$ (the model of $Z$) satisfy the inequality $\deg s^*\calL \leq \alpha (2g(\calC)-2)$, with $\calL$ a model of $L$.
\end{conj}

The isotrivial case of Conjecture \ref{conj:LVFF} is known for subvarieties of abelian varieties \cite{Yamanoi}. For certain cases of Conjecture \ref{conj:LVFF} in the logarithmic setting see \cite{CZConic, CZGm, Tur, advt, CaTur}.

One should notice that for function fields, the height bound predicted by Conjecture \ref{conj:LVFF} can only imply a non-density result when the variety is not isotrivial.
	We note that in the isotrivial case, one can expect that a natural generalization of Conjecture \ref{conj:LVFF} predicts the following
	
	\begin{conj}\label{conj:main_orb}
	  If $(X,\Delta)$ is an orbifold of general type then $(X,\Delta)$ is pseudo algebraically hyperbolic as in Definition \ref{def:ah_orb}. 
	\end{conj}

Motivated by Campana's core construction, as introduced above, we propose the following extension of Conjecture \ref{conj:LVFF} to nonspecial varieties.

\begin{conj}\label{conj:abff}
	If $X$ is not special, then there are a rational dominant map $\pi : X \to Y$ with $\dim Y >0$, a proper algebraic subset $Z \subsetneq \PP(T_X)$, a positive constant $\alpha$ and an ample line bundle $L$ on $Y$ such that sections $s: \calC \to \calX$, whose liftings $s_1: \calC \to \PP(T_\calX)$ are not contained in $\calZ$, satisfy the inequality $\deg (\pi \circ s)^*\calL \leq \alpha (2g(\calC)-2).$
\end{conj}

Finally, in the isotrivial case, a natural generalization of Conjecture \ref{conj:abff} to the logarithmic case predicts the following

\begin{conj}\label{conj:main_ff}
	If $(X,D)$ is not special, then there are a rational dominant map $\pi : (X,D) \to (Y,D') $ with $\dim Y >0$, a proper algebraic subset $Z \subsetneq \PP(T_X)$, a positive constant $\alpha$ and an ample line bundle $L$ on $Y$ such that sections $s: \calC \to X$, whose liftings $s_1: \calC \to \PP(T_X)$ are not contained in $Z$, satisfy the inequality $\deg (\pi \circ s)^*\calL \leq \alpha (2g(\calC)-2 + N_{\pi \circ s}^{[1]}(D')).$
\end{conj}

\section{Special vs weakly special}\label{wspecial}
In this section we construct a family of examples of varieties that are weakly special but are not special in the sense of Definition \ref{def:special}.

\subsection{Weakly special varieties}
We first recall the definition of \emph{weakly special} variety (the terminology is due to Campana).
\begin{definition}
  A smooth projective variety \(X\) over a field $k$ is \textsl{weakly special} if for every finite étale morphism \(u\colon X'\to X_k\) the variety $X'$ does not admit a dominant rational map \(f'\colon X'\to Z'\) to a positive dimensional variety $Z'$ of general type. A projective variety $X$ is weakly special if some (hence any) desingularization is weakly special.
\end{definition}

The Weak Specialness Conjecture, whose original idea is linked to Abramovich and Colliot-Th\'el\`ene in \cite{HT}, predicts that weak specialness characterizes potential density of rational points for varieties defined over number fields.

\begin{conj}[{{see \cite[Conjecture 1.2]{HT}}}]
  \label{ACT}	
  Let $X$ be a projective variety defined over a number field.
Then, the set of rational points on $X$ is potentially dense if and only if $X$ is weakly special.
\end{conj}

Conjecture \ref{ACT} is still open: in fact there is no example of a weakly special variety whose rational points are not potentially dense. The following conjecture is the analogue of Conjecture \ref{ACT} for entire curves and function fields.

\begin{conj}\label{ACT-ff}
  Let $X$ be a projective variety defined over $\C$. Then, $X$ is weakly special if and only if:
  \begin{enumerate}
    \item there exists an entire curve $\C \to X$ with Zariski dense image.
    \item there are no dominant map $\pi: X \to Y$, an ample line bundle $\calL$, a positive constant $\alpha$ and a proper closed subset $Y_\text{exc}$ such that for every smooth integral curve $\calC$ and morphism $s: \calC \to X$ such that $s(\calC)$ is not contained in $Y_{\text{exc}}$ the following holds:
      \[
	\deg (\pi \circ s)^* \calL \leq \alpha \left( 2g(\calC) - 2 \right).
      \]
  \end{enumerate}
\end{conj}

As seen in Section \ref{special}, if a variety $X$ is special then $X$ is weakly special. However, the two notions are equivalent only for curves and surfaces. We construct here examples of \(3\)-dimensional projective varieties which are weakly special, but not special. In Section \ref{EC} we will show that these examples contradict Conjecture \ref{ACT-ff}.

\subsection{Examples of weakly special, but not special threefolds}
The construction below is a slight extension of the construction in \cite{BT} and was explained to us by F. Campana. We shall construct simply-connected smooth projective threefolds \(X\) having no rational fibrations onto varieties of general type, but having equidimensional (elliptic) fibrations of (orbifold) general type \(F\colon X\to S\) onto smooth surfaces \(S\).
These are thus examples of weakly special but not special varieties in the lowest dimension where the two notions do not agree.
\begin{theorem}
  \label{tbt} 
Let $m$ be a positive integer, and let $T,S$ be two surfaces together with fibrations \(f: T \to \PP^1\) and \(g: S \to \PP^1\) such that:
  \begin{enumerate}
    \item
      $T$ is a smooth surface and the fibration \(f\colon T\to\PP^{1}\) has a single multiple fiber \(f^{-1}(0)\eqqcolon m\cdot T_0\), with \(T_0\) a smooth elliptic curve, and another (singular) simply-connected fiber;
    \item
      $S$ is a smooth surface and the fibration \(g\colon S\to\PP^{1}\) has a smooth fiber \(S_{0}\bydef g^{-1}(0)\) such that:
      \begin{enumerate}
        \item
          the surface \(S\) is not of general type
          but
          the orbifold surface 
          \[
            (S,\Delta)
            \bydef
            (S,(1-1/m)\cdot S_{0})
          \]
          is of general type;
        \item
          the complement of \(S_{0}\) in \(S\) is simply connected.
      \end{enumerate}
  \end{enumerate}

  Furthermore, let $X$ be the normalization of the total spaces of the natural (orbifold) elliptic fibrations with equidimensional fibers defined by \(f\) and \(g\), i.e. 
  \[
    X
    \bydef(
    S\times_{\PP^1}T)^\nu
    \stackrel{F}{\longrightarrow}
    (S,\Delta).
  \]
 Then $X$ is a weakly special projective smooth threefold that is not special. 
\end{theorem}
\begin{proof}
  We first notice that the orbifold base of the fibration \(F\) is indeed \(\Delta\):
  by construction, the only multiple fibers lie above \(S_{0}\), and \(F^{\ast}S_{0}=m\cdot(S_{0}\times T_{0})\).
  Since by assumption \((S,\Delta)\) is a surface of general type, it follows immediately that \(X\) is not of special type.

  We now prove that \(X\) is weakly special. A key feature of \(X\) in this direction is simple-connectedness.
  Let \(D\bydef F^{-1}(S_{0})\); then, the fibration \(F\colon(X\setminus D)\to(S\setminus S_{0})\) is a fibration without multiple fibers 
  (since \(f\colon(T\setminus T_0)\to \PP^1\setminus\Set{0}\) has no multiple fibers),
  and with a simply-connected fiber. 
  This implies that \(F_*\colon\pi_1(X\setminus D)\to \pi_1(S\setminus S_{0})\) is an isomorphism. 
  Using our assumption, the group \(\pi_1(X\setminus D)\) is hence trivial.
  This implies that $X$ is simply-connected since the natural map \(\pi_1(X\setminus D)\to \pi_1(X)\) is surjective.

  As a consequence, to prove that \(X\) is weakly special, it suffices to show that no fibration \(h\colon X\to Z\) exists, with \(Z\) of general type and of dimension \(d,\) with \( 0<d\leq 3\), since $X$ does not admit any non-trivial \'etale cover.
  Assume by contradiction that such an \(h\) exists. 
  Then \(d<3\), because \(X\) is not of general type, since it is an elliptic fibration over \(S\); hence \(d=1,2\). 
  Note that, since \(X\) is simply connected, \(Z\) has to be simply connected, thus \(d>1\), since the only simply-connected curve is \(\PP^1\), which is not of general type. 
  We are reduced to the case in which \(d=2\), and \(h\neq F\), since by assumption \(S\) is not of general type. 
  Since, by construction, the fibers of $F$ are special, it follows from \cite[Theorem 2.7]{Ca04} that there exists a map $g: S \to Z$. But since $S$ is not of general type this contradicts the assumption that $Z$ is a surface of general type.
\end{proof}

We now give examples of fibrations $T \to \PP^1$ and $S \to \PP^1$ that satisfy the hypothesis of Theorem \ref{tbt}. In particular the examples show that such fibrations exist \emph{for every} $m \geq 2$ and therefore provide examples of countably many weakly special threefolds $X$ as a corollary of Theorem \ref{tbt}.

\begin{example}[First fibration]\label{ex:1fib}
  In order to construct \(f\colon T\to\PP^{1}\), we consider first a fibration \(f'\colon T'\to\PP^{1}\) of an elliptic surface having no multiple fiber, a smooth fiber \(T'_0\bydef(f')^{-1}(0)\) and a (singular) simply-connected fiber, and moreover such that \(p_g(T')\bydef h^{0,2}(T')=0\). 
  Then \(T'\) is simply connected. Let \(m>1\) be an integer. The Kodaira logarithmic transform \(f\colon T\to\PP^{1}\) of order \(m\) on the fiber \(T'_0\) of \(f'\) is a new elliptic fibration which replaces \(T'_0\) by a smooth multiple fiber \(m\cdot T_0\) of multiplicity \(m\) over \(0\in \PP^1\), leaving the complements \(T'\setminus T'_0\) and \(T\setminus m\cdot T_0\) isomorphic as elliptic fibrations over \(\PP^1\setminus\Set{0}\) (see for example \cite[Section 3]{BT} and for more details \cite[Section 1.6]{FrieMorg}). Moreover \(p_g(T)=p_g(T')=0\), and \(T\) is still simply connected since \(f\) still has a simply connected fiber. This implies that \(b_1(T)=0\) is even and \(T\) is K\"ahler; since \(p_g(T)=0\), we conclude that \(T\) is projective (we refer to \cite[Section 1.6]{FrieMorg} for the details). 
\end{example}

In the original construction of \cite{BT}, \(S\) was chosen so that \(\kappa(S)=1\).
We shall see that a small variation of the construction permits to chose for \((S,\Delta_F)\) a suitable blow-up of any smooth projective surface \(S'\) with \(\kappa(S')\in \{-\infty,0,1\}\) and \(\Delta_F\) the strict transform of a suitable ample orbifold \(\Q\)-divisor on \(S'\).

\begin{example}[Second fibration]\label{ex:2fib}
  In order to construct \(g\colon S\to\PP^{1}\) (given \(m\)), we consider a simply-connected smooth projective surface \(S'\) that is not of general type, together with an ample and smooth divisor \(C'\), which is a member of a pencil \(C'_t\) of divisors on \(S'\) 
  (whose generic member \(C''\) meets \(C'\) transversally at \(C'^2\) distinct points),
  such that the complete linear system \(\mathcal O_{S'}(C')\) is base-point free,
  and such that \(K_{S'}+(1-1/m)\cdot C'\) is a big divisor on \(S'\).

  For example, one can choose \(S'=\PP^2\), together with \(C'\) a smooth curve of degree \(d\geq4\) if \(m\geq 5\), and \(d\geq7\) if \(m\geq 2\).

  Now choose a second generic member \(C''\) of the linear system \(\vert C'\vert\) meeting transversally \(C'\) at \((C')^2\) distinct points. 
  Let \(\sigma\colon S\to S'\) be the blow-up of \(S'\) at these \((C')^2\) points,
  let \(\Delta\) be the strict transform of \(\Delta'\bydef(1-1/m)\cdot C'\) on \(S\),
  and let \(E\) the exceptional divisor of \(\sigma\). 
  Then the divisor
  \begin{align*}
    K_S+\Delta
    =
    (\sigma^*K_{S'}+E)+(\sigma^*\Delta'-(1-1/m)\cdot E)
    =
    \sigma^*(K_{S'}+\Delta')+(1/m)\cdot E
  \end{align*}
  is big since $\sigma$ is birational and all the varieties are projective; this implies that the orbifold surface \((S,\Delta)\) is of general type.

  Let \(k>0\) be the order of divisibility of \([C']\) in \(\Pic(S')\). 
  By (a version of) Lefschetz theorem,  \(\pi_1(S'\setminus C')\) is the cyclic group of order \(k\) generated by a small loop around \(C'\). 
  But blowing-up a point \(a\) on \(C'\) makes this loop become trivial in $\pi_1(E_a\setminus\{a\}) = \pi_1(\C) = 0$, where \(E_{a}\cong \PP^1\) is the exceptional divisor of the blow-up over \(a\), and \(a\in C'\cap C''\subset S'\). 
  This loop thus becomes homotopically trivial in \((S\setminus S_{0})\), which is thus simply-connected.
\end{example}
\begin{remark}
  The smoothness of \(C'\) is not necessary. 
  One may just assume that \(C'\) is nodal and that \(C''\) meets \(C'\) at smooth points of \(C'\).
\end{remark}

\section{Degeneracy results for surfaces}\label{sec:vojta}

In this section we extend results of Corvaja and Zannier in \cite{CZAnnals} for number fields to the orbifold setting, both over function fields and in Nevanlinna Theory. We prove hyperbolicity and degeneracy results for a class of surfaces as a combination of ideas of Corvaja and Zannier together with the recent method introduced by Ru-Vojta in \cite{RV19}.

In particular we recover the function field and analytic analogues of \cite[Main Theorem]{CZAnnals} in the logarithmic setting. We stress that, even if the arithmetic result for orbifold rational points seems at the moment out of reach, our results give evidence for the arithmetic part of Conjecture \ref{conj:Camp}.
Finally, all these results will be applied in Section \ref{EC} to the nonspecial threefold $X_m$ defined in Section \ref{wspecial}.

\subsection{Function Fields}
In this section we let $\kappa$ be an algebraically closed field of characteristic zero and we assume all the varieties defined over $\kappa$. 
The main theorem of this section will be a consequence of the following statement, whose proof is contained in Section \ref{sec:SMT} (see Theorem \ref{thm:ffconst}). We start by recalling the definition of the $\beta$ invariant (we refer to \cite{RV19} for discussion and properties).

\begin{definition}
	Let $X$ be a complete variety, let $\calL$ be a big line bundle on $X$ and let $D$ be a nonzero effective Cartier divisor on $X$. We define
\begin{equation*}
	\beta(\calL, D) = \liminf_{N\to\infty} \frac{\sum_{m\ge 1}h^0(X, \calL^N(-mD))} {Nh^0(X, \calL^{N})}\;.
\end{equation*}
\end{definition}

Using the constant $\beta$ we can reformulate the main theorem of \cite{RV19} in the function field case. 
However, for our applications we need to obtain an explicit dependence of the constants on the Euler characteristic of the curve. This is obtained in the following theorem, which we state in the constant case, by use of the more explicit \cite[Theorem 1]{Wang} instead of the function field analogue of the Schmidt Subspace Theorem \cite[Main Theorem]{Wang}.

 \begin{theorem}\label{th:ffconst_b}
Let $X\subset \PP^m$ be a projective variety over $\kappa$ of dimension $n$, let $D_1,\cdots,D_q$ be effective Cartier divisors intersecting properly on  $X$, and let $\calL$ be a big line sheaf.   Let $\calC$ be a smooth projective curve over $\kappa$, let $S$ be a finite  set of points on $\calC$ and let $K=\kappa(\calC)$ be the function field of $\calC$.  Then for any $\epsilon>0$, there exist constants $c_1$ and $c_2$, independent of the curve $\calC$ and the set $S$, such that for any map $x=[x_0:\hdots:x_m] :\calC\to X$, where $x_i\in K$, outside the augmented base locus of $\calL$ we have 
 either
 \begin{align*}
   \sum_{i=1}^q \beta(\mathcal L, D_i) m_{D_i,S}(x)\le (1+\epsilon) h_{\calL}(x)+c_1\max\{ 1, 2g(\calC)-2+|S| \},
\end{align*}
or the image of $x$  is contained in  a hypersurface (over $\kappa$) in $\PP^m$ of degree at most $c_2$. 
\end{theorem}

\begin{remark}\leavevmode
	\begin{enumerate}
		\item In the case where $X$ is nonsingular, the condition on the proper intersection is equivalent to general position. We refer to \cite[Definition 2.1]{RV19} for precise statements and properties.
		\item Even if Theorem \ref{th:ffconst_b} is stated in the split case we note that our proof carries over to the non split case almost verbatim. We focus here only on the split case since it is the one relevant for our applications.
		\item The constants $c_1$ and $c_2$ appearing in Theorem \ref{th:ffconst_b} can be effectively computed given $X,D_i,\calL,\varepsilon$ so in particular the algebraic hyperbolicity bounds can be made effective in the same way.
	\end{enumerate}
\end{remark}

The proof of Theorem \ref{th:ffconst_b} is included in Section \ref{sec:SMT} as Theorem \ref{thm:ffconst}, and we refer to Subsection \ref{subsec:ffRV} for the definitions of all the quantities involved. For our application we will use the following corollary in dimension 2, whose proof is also included in Section \ref{sec:SMT} (see Corollary \ref{cor:surf}).

\begin{corollary}
  \label{cor:surf_b}
  In the previous setting, if $X$ has dimension 2, then for any $\epsilon>0$, there exist constants $c_1$ and $c_3$ independent of the curve $\calC$ and the set $S$ such that for any $K$-point $x=[x_0:\cdots:x_m] :\calC\to X$, with $x_i\in K$, we have that, either $\deg x(\calC) \leq c_3$, or 
 \begin{align*}
   \sum_{i=1}^q \beta(\mathcal L, D_i) m_{D_i,S}(x)\le (1+\epsilon) h_{\mathscr L}(x)+c_1 \max \{ 1, 2g(\calC)-2+|S| \}.
\end{align*}
\end{corollary}

We now state the main result of this section in the logarithmic case.

\begin{theorem}
		 \label{th:CZ_ff}
		 Let $X\subset\mathbb P^m$ be a smooth projective surface and $D = D_1 + \dots + D_r$ be a divisor with $r \geq 2$, both defined over $\kappa$, such that
		 \begin{enumerate}
			 \item No three of the components $D_i$ meet at a point;
			 \item There exists a choice of positive integers $p_i$ such that 
				 \begin{itemize}
					 \item the divisor $D_p := p_1 D_1 + p_2 D_2 + \dots + p_r D_r$ is ample;
					 \item The following inequality holds:
						 \[
							 2 D_p^2 \xi_i > (D_p \cdot D_i) \xi_i^2 + 3 D_p^2 p_i,
						 \]
						 for every $i = 1,\dots,r$ where $\xi_i$ is the minimal positive solution of the equation $D_i^2 x^2 - 2(D_p \cdot D_i) x + D_p^2 = 0$.
				 \end{itemize}
		 \end{enumerate}
		 Then, $(X,D)$ is pseudo algebraically hyperbolic (see Definition \ref{def:alghyp}).
	 \end{theorem}

	 The arithmetic analogue of Theorem \ref{th:CZ_ff} was proven by Corvaja and Zannier in \cite[Main Theorem]{CZAnnals}. Our proof is different from the arithmetic case since it relies on Theorem \ref{th:ffconst_b} while Corvaja and Zannier's method relies on a direct application of the Schmidt's Subspace Theorem (and, as noted by Vojta, their proof does not apply directly to the function field case). Nevertheless we will use some of the techniques of \cite{CZAnnals} in computing the constants $\beta$.

\begin{lemma}
	In the same setting as Theorem \ref{th:CZ_ff}, for every $i=1,\dots,r$, $\beta_i = \beta(D_p,D_i) > p_i$.
	\label{lemma:pi}
\end{lemma}

\begin{proof}
	By the Riemann-Roch Theorem, for $N$ large enough we have that $ 2 h^0(ND_p)= D_p^2N^2+O(N)$. In order to compute $\beta_i$, we consider the divisor $ND_p -m D_i$ which is effective if $m\le N p_i$.

By the same computation as in \cite{CZAnnals} (in particular see pages 718--719) we get
\begin{equation}
	\sum_{m=0}^{\xi_i N} h^0(ND_p - mD_i) \geq N^3 \left( \frac{\xi_i^2 ( D_p \cdot D_i)}{2} - \frac{\xi_i^3 D_{i}^2}{3} \right) + O(N^2).
	\label{eq:CZl}
\end{equation}

Recall that by definition of $\xi_i$, we have $D_i^2 \xi_i^3 = 2(D_p \cdot D_i) \xi_i^2 - D_p^2 \xi = 0.$ Using this in equation \ref{eq:CZl} we obtain

\begin{align*}
	2\sum_{m=0}^{\xi_i N}h^0(NA-m  D_i)&\ge 2 N^3 \left( \frac{\xi_i^2 (D_p \cdot D_i)}{2} - \frac{2 \xi_i^2 (D_p \cdot D_i)}{3} + \frac{D_p^2 \xi_i}{3} \right) + O(N^2) \\
					&= \left(\frac23  \xi_iD_p^2-\frac13(D_p \cdot D_i)\xi_i^2\right)N^3        + O(N^2).
\end{align*}
This implies that
\begin{align*}
\beta_i=\frac{ \frac23  \xi_iD_p^2-\frac13(D_p \cdot D_i)\xi_i^2}{D_p^2}>p_i.
 \end{align*}
\end{proof}

We also include an easy lemma which provides a useful trick to get a lower bound of the height of a point with respect to $D_p$.

\begin{lemma}
  In the same settings as Theorem \ref{th:CZ_ff}, let $\varepsilon = \min\{ (\beta-p_i)/p_i  \}$, then for very $\varphi: \calC \to X$,
	\[
		\sum_{i=1}^r \beta_i h_{D_i}(\varphi) \geq (1 + \varepsilon) h_{D_{p}}(\varphi).
	\]
	\label{lemma:epsilon}
\end{lemma}
\begin{proof}
	We have that for every $i$ 
	\[
		\beta_i h_{D_i}(\varphi) = \left( \frac{\beta_i - p_i}{p_i} + 1 \right) p_i h_{D_i}(\varphi).
	\]
	Then, summing over $i=1,\dots,r$ we get
	\[
		\sum_{i=1}^r \beta_i h_{D_i}(\varphi)  = \sum_{i=1}^r \left( \frac{\beta_i - p_i}{p_i} + 1 \right) p_i h_{D_i}(\varphi) \geq (1 + \varepsilon) h_{D_{p}}(\varphi).
	\]
\end{proof}

\begin{proof}[Proof of Theorem \ref{th:CZ_ff}]
	We will prove that there exists a constant $\alpha > 0$ such that for every morphism $\varphi: \calC \to X$ with $\varphi(\calC) \nsubseteq D$ one has
	\begin{equation}
		\deg \varphi^*D_p \leq \alpha \cdot \max \{ 1, 2g(\calC) - 2 + N^{[1]}_\varphi(D_p) \}.
	  \label{eq:mainboundff}
	\end{equation}
	We note that such a bound implies that $(X,D)$ is pseudo algebraically hyperbolic in the sense of Definition \ref{def:alghyp} provided that the union of images of curves for which $2g(\calC) - 2 + N^{[1]}_\varphi(D_p) \leq 0$ is a proper closed subset of $X$. We will first prove the bound \eqref{eq:mainboundff} and then show that the exceptional set is indeed a proper and closed subset.
First we note that condition (1) implies that $D_1,\hdots,D_r$ are in general position.  
We fix
\begin{align*}
	\varepsilon=\min_i\left\{\dfrac{\beta_i-p_i}{p_i}\right\}
\end{align*}
which is positive by Lemma \ref{lemma:pi}. Let  $\varphi: \calC \to X$  be a morphism such that $\varphi(\calC) \nsubseteq D$.
By Corollary \ref{cor:surf_b}, applied with $\calL = \calO_X(D_p)$, there exist constants $c_1$ and $c_3$ depending only on $X$ and $\varepsilon$ (independent of $\calC$ and $\varphi$) such that either $\varphi(\calC)$ has degree bounded by $c_3$ or
\begin{align}\label{mainineq2}
	\sum_{i=1}^r \beta_i m_{D_i,S}(\varphi)\le (1+\frac{\epsilon}2) h_{D_p}(\varphi)+c_2 \max\{ 1, 2g(\calC) - 2 + \lvert S \rvert \},
\end{align}
where $S = \supp(\varphi^* D)$. In the first case we are done. In the latter, since the support of $\varphi^* D$ is contained in $S$, $m_{D_i,S}(\varphi)=h_{D_i}(\varphi) + O(1)$.

Lemma \ref{lemma:epsilon} gives the lower bound
\begin{align*}
	\sum_{i=1}^r \beta_im_{D_i,S}(\varphi)> (1+\epsilon) h_{D_p}(\varphi).
\end{align*}
Together with \eqref{mainineq2}, this implies that there exists a constant $c_2'$ such that 
\[
	\frac  {\epsilon}2 h_{D_p}(\varphi)< c_2' \max\{ 1, 2g(\calC) - 2 + \lvert S \rvert  \}.
\]
Hence, 
\[
	\deg \varphi^* D_p = h_{D_p}(\varphi) < 2c_2' \epsilon^{-1}\max\{ 1, 2g(\calC) - 2 + N_\varphi^{[1]}(D_p) \}.
\]
To finish the proof we need to show that there exists a closed subvariety $Z$ that contains all images $\varphi(\calC)$ when $2g(\calC) - 2 + N_\varphi^{[1]}(D_p)$ is not positive. First of all, we can reduce to the case in which $\kappa = \C$: it is enough to notice that given $(X,D)$ defined over $\kappa$ there exists a field $\kappa_0$ that is an algebraically closed subfield of $\kappa$, which is finitely generated over $\Qbar$, and such that $(X,D)$ has a model over $\kappa_0$. For such $\kappa_0$ there exists an embedding $\kappa_0 \to \C$, and therefore we can reduce the problem to the case in which $(X,D)$ is defined over $\C$.

In this case, we claim that the exceptional set $Z$ can be chosen to be the exceptional set in \cite[Theorem 8.3 B]{Levin}. In fact, the same strategy of Levin applies to the setting of our theorem (in particular the proof of \cite[Main Theorem]{CZAnnals} shows that $D_p$ is \emph{large} in the sense of Levin) and therefore there exists a proper closed subvariety $Z \subset X$ that contains all the images of entire curve $\C \to X$. To conclude, it is sufficient to notice that for every curve $\calC$ such that $2g(\calC) - 2 + N_\varphi^{[1]}(D_p) \leq 0$ there exists a non trivial holomorphic map $\C \to \calC \setminus \supp \varphi^* D_p$. In particular this implies that $\varphi(\calC)$ has to be contained in $Z$ as wanted.

\end{proof}
	
Next we show that Theorem \ref{th:CZ_ff} leads to a statement for orbifold morphisms, in particular we obtain that Conjecture \ref{conj:main_orb} holds in this setting.

\begin{corollary}
	Let $(X,D)$ as in Theorem \ref{th:CZ_ff}. Let $\Delta$ be the $\Q$-divisor defined as
	\[
		\Delta = \sum_{i=1}^r \left( 1 - \frac{1}{m_i} \right) D_i,
	\]
	for some integers $m_i \geq 1$.
	Then, there exists a positive integer $m > 0$ such that, if $m_i > m$ for every $i$, $(X,\Delta)$ is pseudo algebraically hyperbolic (see Definition \ref{def:ah_orb}).
	\label{cor:CZ_orb}
\end{corollary}

We obtain the previous corollary applying the following general lemma to Theorem \ref{th:CZ_ff}. Note that, given an orbifold $(X,\Delta)$ such that $X \setminus \supp(\Delta)$ is of log-general type, there exists an integer $m > 0$ such that, if all orbifold multiplicities of $\Delta$ are greater than $m$, the orbifold $(X,\Delta)$ is of general type. In particular the following lemma is in accordance with Conjecture \ref{conj:main_orb}.

\begin{lemma}
\label{lemma:orbif_ff}
	Let $(X,\Delta)$ be an orbifold defined over $\kappa$ such that $(X,\supp(\Delta))$ is pseudo algebraically hyperbolic with a constant $\alpha>0$ and with exceptional set $Z \subset X$.
	Then, there exists a positive integer $m > 0$ such that, if all orbifold multiplicities are greater than $m$, $(X,\Delta)$ is pseudo algebraically hyperbolic with the same constant $\alpha$ and the same exceptional set $Z$.
\end{lemma}
	
We note that Lemma \ref{lemma:orbif_ff} and its proof generalize \cite[Theorem 4.8]{Rou} where it was assumed that $X = \PP^n$.
	
\begin{proof}[Proof of Lemma \ref{lemma:orbif_ff}]
	Let $\psi: (\calC, \Delta_C) \to (X,\Delta)$ be an orbifold morphism. We will use the notation $\delta: = N_\psi^{[1]}(\supp(\Delta))$. 
	Let 
\begin{eqnarray*}
\psi^*(\Delta_{j}) & = & \sum_{i=1}^{\delta}t_{i,j}P_{i},\\
\psi^*(\supp(\Delta)) & = & \sum_{i=1}^{\delta}t_{i}P_{i},\\ 
\end{eqnarray*}
where the $P_{i}$ are the distinct points of $\psi^{-1}(\supp(\Delta))$. Then if 
\[
	\Delta_C=\sum_{i=1}^{\delta}\left( 1 - \frac{1}{m_{i}'}\right) P_{i},
	\]
	$\psi$ is an orbifold morphism if $m_{i}'t_{i}\geq m_{j}$ for all $j\in \varphi(i)$ where $\varphi(i)$ is the set 
\[
	\varphi(i) :=\{ k: 1 \leq k \leq \delta \text{ and } \psi(P_{i}) \in \Delta_{k}\} .
\]
We define the orbifold structure $\widetilde{\Delta}$ on $\calC$ by the multiplicities
\[
	\widetilde{m_{i}}=\sup_{j\in \varphi(i)} \left\lceil\frac{m_{j}}{t_{i}}\right\rceil,
\]
where as usual $\lceil k \rceil$ denotes the round up of $k$. By definition $\psi: (C,\widetilde{\Delta}) \to (X,\Delta)$ is an orbifold morphism. Moreover, for every orbifold structure $\Delta'$ on $\calC$ such that $\psi:(C,\Delta') \to (X,\Delta)$ is an orbifold morphism we have $\Delta' \geq \widetilde{\Delta}$. In particular it is enough to prove that $\psi$ satisfies a degree bound as in Definition \ref{def:ah_orb} for the orbifold structure $\widetilde{\Delta}$. In other words we have to prove that there exists a constant $\alpha_\Delta$ and an ample line bundle $\calL_\Delta$ on $X$ such that, for every $\psi$ such that $\psi(\calC) \nsubseteq Z$ we have
\[
	\deg \psi^*\calL_\Delta \leq \alpha_\Delta \left(2g(C)-2+\sum_{i=1}^{\delta} \left(1-\frac{1}{\widetilde{m_{i}}}\right)\right). 
\]
We can bound the contribution coming from the orbifold divisor as follows: 
\begin{equation}
	\sum_{i=1}^{\delta} \left(1-\frac{1}{\widetilde{m_{i}}}\right) \geq \sum_{i=1}^{\delta}\left(1-\dfrac{t_{i}}{\displaystyle \sup_{j\in\varphi(i)}m_{j}}\right)\geq N_\psi^{[1]}(\supp(\Delta)) -\sum_{i=1}^{\delta}\sum_{j=1}^{q}\frac{t_{i,j}}{m_{j}} 
	\label{eq:orb_bound}
\end{equation}
On the other hand, by definition
\[
	\sum_{i=1}^{\delta}t_{i,j}=\deg(\psi^*\Delta_{j}),
\]
which implies that we can rewrite \eqref{eq:orb_bound} as
\begin{equation}
	\sum_{i=1}^{\delta} \left(1-\frac{1}{\widetilde{m_{i}}}\right) \geq N_\psi^{[1]}(\supp(\Delta)) - \sum_{j=1}^{q} \frac{\deg(\psi^*\Delta_{j})}{m_j}.
  \label{eq:orb_bound2}
\end{equation}
In particular, equation \eqref{eq:orb_bound2} can be rewritten as
\begin{equation}
	N_\psi^{[1]}(\supp(\Delta)) \leq \sum_{i=1}^{\delta} \left(1-\frac{1}{\widetilde{m_{i}}}\right) + \sum_{j=1}^{q} \frac{\deg(\psi^*\Delta_{j})}{m_j}.
	\label{eq:N1orb}
\end{equation}

On the other hand, our assumption that $(X,\supp(\Delta))$ is pseudo algebraically hyperbolic, implies that, for every ample line bundle $\calL$ on $L$, if $\psi(\calC) \nsubseteq Z$, we have
\begin{equation}
	\deg \psi^* \calL \leq \alpha\left( 2g(\calC) - 2 + N_\psi^{[1]}(\supp \Delta) \right).	
	\label{eq:log_hyp}
\end{equation}

Finally, combining equation \eqref{eq:log_hyp} with equation \eqref{eq:N1orb}, we obtain
\[
	\deg \psi^* \calL \leq \alpha\left( 2g(\calC)-2+ \deg \widetilde{\Delta} \right) + \deg \psi^* \left( \sum_{j=1}^{q} \frac{\alpha}{m_j}\Delta_j\right).
\]
Therefore we conclude by noticing that, when the multiplicities $m_j$ are big enough, the line bundle 
\[
	\calL_\Delta = \calL \otimes \calO_X(-\sum_{j=1}^\delta \frac{\alpha}{m_j} \Delta_j)
\]
is ample on $X$ and therefore we obtain
\[
\deg \psi^* \calL_\Delta \leq \alpha\left( 2g(\calC)-2+ \deg \widetilde{\Delta} \right). 
\]
\end{proof}
Finally we can prove Corollary \ref{cor:CZ_orb}.
\begin{proof}[Proof of Corollary \ref{cor:CZ_orb}]
	The conclusion follows from Theorem \ref{th:CZ_ff} and Lemma \ref{lemma:orbif_ff}.
\end{proof}

\subsection{Holomorphic Maps}
In this subsection we are interested in degeneracy properties of holomorphic maps. All varieties will be defined over $\C$. As before we will obtain our result as a consequence of the following generalization of \cite[General Theorem (Analytic Part)]{RV19} that includes truncation.

 \begin{theorem}\label{th:trungeneral}
    Let $X$ be a complex projective variety of dimension $n$ and let $D_1,\cdots,D_q$ be effective Cartier divisors intersecting properly on $X$. Let $\calL$ be a big line bundle. For each $\varepsilon>0$,  there exists a positive integer $Q$  such that for any algebraically non-degenerate holomorphic map $f:\mathbb{C}\to X$,  the inequality 
\begin{align*}
    \sum_{j=1}^q\beta(\mathscr L, D_j)T_{D_j,f}(r)-  (1+\varepsilon)T_{\mathscr L,f}(r)\leq_{\operatorname{exc}} \sum_{j=1}^q \beta(\mathscr L, D_j)N_f^{(Q)}(D_j,r) 
\end{align*}
    holds,  where $\leq_{\operatorname{exc}}$ means that the inequality holds for all $r\in \mathbb{R}^+$ except a set of finite Lebesgue measure.
    \end{theorem}

    The proof of Theorem \ref{th:trungeneral} will be given in Section \ref{sec:SMT} (see in particular Theorem \ref{trungeneral}). Using the previous theorem we obtain a Nevanlinna analogue of Theorem \ref{th:CZ_ff} and Corollary \ref{cor:CZ_orb}.

\begin{theorem}\label{th:CZ_nev}
		 Let $X\subset\mathbb P^m$ be a complex nonsingular projective surface and $D = D_1 + \dots + D_q$ be a divisor with $q \geq 2$, such that
		 \begin{enumerate}
			 \item No three of the components $D_i$ meet at a point;
			 \item There exists a choice of positive integers $p_i$ such that 
				 \begin{itemize}
					 \item the divisor $D_p := p_1 D_1 + p_2 D_2 + \dots + p_q D_q$ is ample;
					 \item The following inequality holds:
						 \[
							 2 D_p^2 \xi_i > (D_p \cdot D_i) \xi_i^2 + 3 D_p^2 p_i,
						 \]
						 for every $i = 1,\dots,q$ where $\xi_i$ is the minimal positive solution of the equation $D_i^2 x^2 - 2(D_p \cdot D_i) x + D_p^2 = 0$.
				 \end{itemize}
		 \end{enumerate}
		Let $\Delta$ be the $\Q$-divisor defined as
	\[
		\Delta = \sum_{i=1}^q \left( 1 - \frac{1}{m_i} \right) D_i.
	\]
	Then, there exists a positive integer $m$ such that, if $m_i \geq m$ for every $i$, every orbifold entire curve $\psi: \C \to (X,\Delta)$, is algebraically degenerate.
	 \end{theorem}

	 \begin{proof}
		 First we note that condition (1) implies that $D_1,\hdots,D_q$ are in general position. Let $f: \C \to (X,\Delta)$ be a Zariski dense orbifold entire curve.
	As before, we denote by $\beta_i = \beta(D_p,D_i)$ and we set
	\begin{align*}
	\varepsilon=\min\{(\beta_i-p_i)/p_i\}
	\end{align*}
	which is positive by Lemma \ref{lemma:pi}.
	By Theorem \ref{th:trungeneral}, applied with $\calL = \calO_X(D_p)$, and $\varepsilon/2$, there exists an integer $Q$ such that the following inequality holds:
\begin{align}\label{eq:main_nev}
	\sum_{j=1}^q\beta_j T_{D_j,f}(r)-  \left(1+\dfrac{\varepsilon}{2}\right)T_{D_p,f}(r)\leq_{\operatorname{exc}} \sum_{j=1}^q \beta_j N_f^{(Q)}(D_j,r).
\end{align}

	The analogue of Lemma \ref{lemma:epsilon} in the Nevanlinna setting gives the lower bound
	\begin{align*}
		\sum_{j=1}^q \beta_i T_{D_i,f}(r)> (1+\varepsilon) T_{D_p,f}(r).
	\end{align*}
	Together with \eqref{eq:main_nev}, this implies that 
	\begin{equation}
		\frac  {\epsilon}2 T_{D_p,f}(r) \leq \sum_{j=1}^q \beta_j N_f^{(Q)}(D_j,r).
		\label{eq:nev_2}
	\end{equation}
	Since $f$ is an orbifold entire curve, one has 
\[
N_f^{(Q)}(D_j,r) \leq \frac{Q}{m_j} N_f(D_j,r) \leq \frac{Q}{m_j} T_{D_p,f}(r).
\]
Therefore, we can rewrite equation \eqref{eq:nev_2}, as
\begin{equation*}
	\frac{\epsilon}{2} T_{D_p,f}(r) \leq Q  \left( \sum_{j=1}^q \frac{\beta_j}{m_j} \right) T_{D_p,f}(r) \leq \frac{Q}{m} \left( \sum_{j=1}^q \beta_j \right) T_{D_p,f}(r).
\end{equation*}
	
To conclude, it is enough to choose $m$ big enough, such that 
\[
	\frac{Q}{m} \left( \sum_{j=1}^q \beta_j \right) < \frac{\varepsilon}{2}.
\]
	 \end{proof}

	 As a corollary we obtain the analogue of \cite[Main Theorem]{CZAnnals} and Theorem \ref{th:CZ_ff} in the Nevanlinna setting.

	\begin{corollary}
	  Let $(X,D)$ be as in Theorem \ref{th:CZ_nev}. Then, every entire curve $\psi: \C \to X\setminus D$, is algebraically degenerate. Moreover $(X,D)$ is Brody hyperbolic modulo the exceptional set $Z$ of Theorem \ref{th:CZ_ff}, i.e. every entire curve $\psi: \C \to X\setminus D$ verifies $\psi(\C)$ is contained in $Z$.
	 \end{corollary}
	 \begin{proof}
	   The first statement follows directly by Theorem \ref{th:CZ_nev}. For the second statement one can apply directly \cite[Theorem 8.3 B]{Levin} to our setting (since, as noted before, the proof of \cite[Main Theorem]{CZAnnals} shows that the divisor $D_p$ is \emph{large} in the sense of Levin).
	 \end{proof}

\section{Degeneracy properties of \texorpdfstring{$X_m$}{Xm}}\label{EC}

Given Lang and Vojta's dictionary between arithmetic and geometry properties of a variety $X$, it is expected that varieties with a potentially dense set of rational points should correspond to manifolds admitting Zariski dense entire curves. Therefore the analogue of Conjecture \ref{ACT} should imply that weakly special manifolds should admit such curves (see Conjecture \ref{ACT-ff}).

Campana and P\u{a}un \cite{CP} have shown that some examples constructed in \cite{BT} (in particular with \(\kappa(S)=1\)) give counterexamples to such a statement. In other words, there are some weakly special manifolds in which all entire curves are algebraically degenerate. The goal of this section is to show that one can produce many more ``counterexamples'' from the examples given in the previous section. In particular we show that the weakly special varieties $X_m$ (for $m$ big enough) provide counterexamples to Conjecture \ref{ACT-ff}.

\subsection{Construction}

Following ideas of Corvaja and Zannier in \cite{CoZa10} we begin by giving a series of examples in which the results of Section \ref{sec:vojta} apply. As in Section \ref{sec:vojta} we denote by $\kappa$ an algebraically closed field of characteristic 0. We start with the following definition.

\begin{definition}
  \label{def:gen_position}
  Let $D_1, D_2$ and $D_3$, be three projective plane curve in $\PP^2_\kappa$ and let $D_4 = H$ be an hyperplane such that $D_i\cap D_j \cap D_k$ is empty for every distinct $i,j,k$.
  Given three distinct projective curves $B_1,B_2,B_3$ we say that $(D_1,B_1),(D_2,B_2),(D_3,B_3)$ and $H$ are in \emph{general position} if 
  \begin{itemize}
	  \item $D_i$ and $B_i$ intersect transversally for every $i=1,2,3$;
	  \item $D_i \cap D_j \cap B_h = \emptyset$, for every distinct $i, j$ and $h\in\{i,j\}$.
  \end{itemize}

\end{definition}

The main source of examples will be certain blow-ups of $\PP^2$. In the next proposition we show that these satisfy the assumptions of Theorem \ref{th:CZ_ff} and Theorem \ref{th:CZ_nev}.

\begin{proposition}
	\label{prop:blow_up}
	Let $(D_1,B_1),(D_2,B_2),(D_3,B_3)$ and $H$ be curves in general position as in Definition \ref{def:gen_position} such that $\deg D_i \geq \deg B_i$. Let $T$ be the set of points $D_i \cap B_i$ for $i=1,2,3$; if $\# (T \cap D_i) < \deg D_i^2$ we add to $T$ smooth points of $D_i$ so that $\# T = \deg D_1^2 + \deg D_2^2 + \deg D_3^2$.

	Let $X$ be the blow up $\PP^2$ along $T$, and denote by $\widetilde{D}_i$ the strict transform of $D_i$, and by $\widetilde{H}$ the strict transform of $H$. Let $\Delta$ be the $\Q$-divisor defined as
	\[
		\Delta = \sum_{i=1}^3 \left( 1 - \frac{1}{m_i} \right) \widetilde{D}_i + \left( 1 - \frac{1}{m_4} \right) \widetilde{H}.
		\]
		Then, there exists a positive integer $m > 0$ such that, if $m_i > m$, we have 
		\begin{enumerate}
			\item $(X,\Delta)$ is pseudo algebraically hyperbolic;
			\item if $\kappa = \C$, every orbifold entire curve $f: \C \to (X,\Delta)$ is degenerate.
		\end{enumerate}
\end{proposition}

	 \begin{proof}[Proof of Proposition \ref{prop:blow_up}]
		 Let $d_i = \deg D_i$. By definition, we have the following intersection numbers in $X$: $\widetilde{D}_i \cdot \widetilde{D}_j = d_i d_j$, $\widetilde{D}_i^2 = 0$ and $\widetilde{D}_i \cdot \widetilde{H} = d_i$. The result will follow from an application of Theorem \ref{th:CZ_ff}.
	Let $c := 4 d_1 d_2 d_3$ and define $p_i := c/d_i$ for $i=1,2,3$,  and $p_4 := 3c/4 = 3 d_1 d_2 d_3$. Using the integers $p_1,p_2,p_3$ and $p_4$ we define
\begin{equation*}
	D_p = p_1 \widetilde{D}_1 + p_2 \widetilde{D}_2 + p_3 \widetilde{D}_3 + p_4 \widetilde{H}.
	\label{eq:D_p}
\end{equation*}
Then, it is immediate to verify that $D_p \cdot \widetilde{H} > 0$ and $D_p \cdot E > 0$ where $E$ is any exceptional divisor. Now let $C$ be a non-exceptional curve in $X$. Then we can compute
\begin{equation}
  D_p \cdot C = \left(\sum_{i=1}^3 p_i d_i + p_4 \right) H \cdot \pi_* C - \sum_{i=1}^3 p_i \sum_{Q \in D_i \cap T} \mult_{Q}(\pi_* C).
	\label{eq:inters}
\end{equation}
Then, denoting $c = \deg \pi_* C$, Bezout's Theorem implies that $\sum_{Q \in D_i \cap T}\mult_{Q}(\pi_* C) \leq c d_i$ and therefore we can rewrite equation \eqref{eq:inters} as
\[
	D_p \cdot C \geq c \left(\sum_{i=1}^3 p_i d_i + p_4\right) - \sum_{i=1}^3 p_i (c d_i) = cp_4 > 0.
\]
Finally, since $D_p^2 >0$, by the Nakai-Moishezon criterion \cite[Theorem 1.2.23]{Lazarsfeld1}, the divisor $D_p$ is ample.

	Let $i=1,2,3$, and let $\xi_i$ be the smallest solution of
	\begin{equation}
		\widetilde{D}_i^2 x^2 - 2(D_p \cdot \widetilde{D}_i) x + D_p^2 = 0.
		\label{eq:def_xi}
	\end{equation}
	Using the fact that $\widetilde{D}_i^2 = 0$ for $i=1,2,3$ this gives the following expression
	\[
		\xi_i = \dfrac{D_p^2}{2(D_p \cdot D_i)}.
	\]
	Using this expression, in order to apply Theorem \ref{th:CZ_ff} and Theorem \ref{th:CZ_nev}, we have to verify that the following inequality holds true:
	\begin{equation}
		2 D_p^2 \cdot \dfrac{D_p^2}{2(D_p \cdot \widetilde{D}_i)} > (D_p \cdot \widetilde{D}_i) \cdot \dfrac{D_p^4}{4(D_p \cdot \widetilde{D}_i)^2} + 3 D_p^2 \cdot p_i
		\label{eq:main_ineq}
	\end{equation}
	This simplifies to $D_p^2 > 4p_i(D_p \cdot \widetilde{D}_i)$.
Using the definition of $D_p$ and $p_i$ we can compute $D_p^2 = 177 d^2$ where $d = d_1 d_2 d_3$, and $D_p \cdot \widetilde{D}_i/d_i = 11 d$. In particular $4p_i(D_p \cdot \widetilde{D}_i) = 176 d^2$ so that \eqref{eq:main_ineq} reads $177 d^2 > 176 d^2$ and therefore it is verified in the case in which $i=1,2,3$.\medskip

To conclude we need to verify that the same condition holds for $i=4$. In this case, using that $D_p \cdot \widetilde{H} = 15 d$, equation \eqref{eq:def_xi} becomes
\begin{equation*}
	x^2 - 30 d\, x + 177 d^2 = 0
	\label{eq:def_xi_4}
\end{equation*}
A direct computation shows that the smallest solution of the equation is $\xi = d(15-4\sqrt{3})$. Then, inequality \eqref{eq:main_ineq} becomes
	\[
		2 \cdot 177 \cdot (15 - 4 \sqrt{3}) > 15 \cdot (15 - 4\sqrt{3})^2 + 9 \cdot 177
	\]
that is a true statement, thus concluding the verification of the hypotheses of Theorem \ref{th:CZ_ff} and Theorem \ref{th:CZ_nev} which imply the desired conclusion.
\end{proof}

\begin{remark}
	We note that the orbifold $(X,\Delta)$ will be of general type as soon as the $m_i$'s are big enough.
	Moreover, in the case in which all the multiplicities $m_i = \infty$, we recover the analogue of \cite[Proposition 2]{CoZa10} in the function field and Nevanlinna case.
\end{remark}

\subsection{Weakly-Special Threefolds}

In order to prove degeneracy results for the threefolds $X_m$ constructed in Section \ref{wspecial}, we will use Proposition \ref{prop:blow_up} in a special case, that guarantees that the quasi-projective surface $S$ is not of log general type but the orbifold $(S,\Delta)$ is of \emph{orbifold} general type.

Let $D:=D_1+D_2+D_3+D_4$ be a simple normal crossing divisor on $\PP^2$ where $D_1$ is a curve of degree $d$ and $D_2, D_3, D_4$ are lines. Let $F:=F_1+F_2+F_3$ be a simple normal crossing divisor on $\PP^2$ where $F_1$ is a curve of degree $d$ and $F_2, F_3$ are lines and such that $F+D$ is a simple normal crossing divisor (so that $(D_i,F_i)$ for $i=1,2,3$ and $D_4$ are in general position in the sense of Definition \ref{def:gen_position}). Let $T_i:=D_i \cap F_i$, $1 \leq i \leq 3$ and $T=T_1\cup T_2 \cup T_3$. 
Let $\pi: \overline{S} \to \PP^2$ be the blow-up of $\PP^2$ in $T$ and $\widetilde{D}_i = \pi_*^{-1}D_i$ the strict transform of $D_i$ for $i=1,\dots,4$. In order to construct the fibration $g$ of Example \ref{ex:2fib} we consider the quasi-projective surface $S$ given by $\overline{S} \setminus \widetilde{D}_2 \cup \widetilde{D}_3 \cup \widetilde{D}_4$ and the fibration $g: S \to \PP^1$ to be the fibration induced by $D_1$ and $F_1$. Note that for any integer $m\geq 2$, the divisor $K_{\PP^2}+(1-1/m)D_1+D_2+D_3+D_4$ is big, which implies in particular that the orbifold surface $(S, (1-1/m)\widetilde{D}_1)$ is an orbifold of general type. Moreover $S \setminus \widetilde{D}_1$ is simply connected (by the same argument as in Example \ref{ex:2fib}).

\begin{remark}
  The quasi projective surface $\overline{S} \setminus \widetilde{D}_1 \cup \widetilde{D}_2 \cup \widetilde{D}_3 \cup \widetilde{D}_4$ appears in \cite[Theorem 3]{CoZa10} to produce an example of a simply-connected quasi-projective surface with non Zariski-dense set of integral points. The idea of our construction is to consider instead the \emph{orbifold} surface $(\overline{S}, (1-1/m)\widetilde{D}_1+\widetilde{D}_2+\widetilde{D}_3+\widetilde{D}_4),$ and prove the degeneracy of orbifold entire curves. Unfortunately, it seems that the arithmetic orbifold analogue is out of reach with the present methods.
\end{remark}

Given $f$ as in Example \ref{ex:1fib} we can construct the smooth threefold $X_m$ as in Theorem \ref{tbt}: it is a weakly special threefold that is not special. We denote by $\pi : X_m \to S$ the elliptic fibration induced. Note that by construction the orbifold base (see Definition \ref{dob}) of $\pi$ coincides with  $\Delta_\pi = ( 1 - 1/m) \widetilde{D}_1$.  Then, the following theorem proves degeneracy results analogues of Conjecture \ref{conj:main_ff} and the analytic part of Conjecture \ref{conj:Camp} for the \emph{quasi-projective} threefold $X_m$.

\begin{theorem}\label{deg}
	There exists $m_0$ such that for all $m \geq m_0$, the following holds:
	\begin{enumerate}
	  \item there exists a positive constant $A$ such that, for every morphism $\varphi: \calC \to \overline{X}_m$ such that $\pi(\varphi(\calC))$ is not contained in the exceptional set of Proposition \ref{prop:blow_up}, the following holds:
			\[
			  \deg \pi(\varphi(\calC))\leq A \left( 2g(\calC) - 2 + N_{\pi \circ \varphi}^{[1]}(\widetilde{D}_2+\widetilde{D}_3+\widetilde{D}_4)\right).
			\]
		\item every entire curve $f: \C \to X_m$ is algebraically degenerate.
	\end{enumerate}
	In particular, the threefolds $X_m$ give counterexamples to (the logarithmic analogue of) Conjecture \ref{ACT-ff}.
	
\end{theorem}

\begin{proof}
  We first prove the statement in the function field case. Let $\varphi$ be a morphism as above and consider $\pi \circ \varphi: \calC \to \overline{S}$. By construction, this induces an orbifold morphism 
\[
	\left(\calC, \varphi^{-1}(\widetilde{D}_2+\widetilde{D}_3+\widetilde{D}_4)\right) \to \left(\overline{S}, \left(1-\dfrac{1}{m}\right)\widetilde{D}_1+\widetilde{D}_2+\widetilde{D}_3+\widetilde{D}_4\right).
\]
Then the conclusion follows from Proposition \ref{prop:blow_up}.
	Similarly, in the Nevanlinna case, one considers an entire curve $f: \C \to X_m$: the composition $\pi \circ f$ is an orbifold entire curve to $(\overline{S}, (1-1/m)\widetilde{D}_1+\widetilde{D}_2+\widetilde{D}_3+\widetilde{D}_4).$ Then the conclusion follows again from Proposition \ref{prop:blow_up}.
\end{proof}

\section{The Ru-Vojta Method}\label{sec:SMT}
The purpose of this section is to obtain a truncated version of Ru-Vojta's Theorem \cite[General Theorem (Analytic Part)]{RV19}, as well as its analogue in the function field setting.  
 \subsection{Nevanlinna Theory}
  We first recall some definitions in Nevanlinna theory.
  Let $D$ be an effective Cartier divisor on a complex variety $X$.  Let  $s=1_D$ be a canonical section of $\calO(D)$, i.e. a global section for which $(s)=D$, and fix a smooth metric $|\cdot|$
on $\calO(D)$.  The associated Weil function $ \lambda_D: X(\C)\setminus\Supp D \to \R$ is given by
\begin{equation*} 
  \lambda_D(x) := -\log|s(x)|\;.
\end{equation*}
It is linear in $D$ (over a suitable domain), so by linearity and continuity
it can be defined for a general Cartier divisor
$D$ on $X$.

Let $f:\C\to X$ be a holomorphic map whose image is not contained in the support of divisor $D$ on $X$.
The {\it proximity function} of $f$ with respect to $D$ is defined by 
\[
  m_f(D,r)=\int_0^{2\pi}\lambda_D(f (re^{i\theta}))\frac{d\theta}{2\pi}.
\]
Let $n_f(D,t)$ (respectively, $n^{(Q)}_f(D,t)$) be the number of zeros of $\rho\circ f$ inside $\{|z|<t\}$, counting multiplicity, (respectively, ignoring multiplicity larger than $Q\in\mathbb N$) with $\rho$ a local equation of $D$.
The  {\it counting function} and the {\it truncated counting function} of $f$ of order $Q$ at $\infty$  are  defined, respectively, by
\[
  N_f(D,r) = \int_1^r \dfrac{n_f(D,t) }{t} dt \qquad \text{and} \qquad N^{(Q)}_f(D,r) = \int_1^r \dfrac{n^Q_f(D,t) }{t} dt.
\]
        The height function relative to $D$ is defined, up to $O(1)$, as 
\begin{align}\label{FMT}
T_{D,f}(r)=m_f(D,r)+N_f(D,r).
\end{align}
 
The following is a truncated version of the analytic part \cite[General Theorem]{RV19}.
 
 \begin{theorem}\label{trungeneral}
    Let $X$ be a complex projective variety of dimension $n$, let $D_1,\cdots,D_q$ be effective Cartier divisors intersecting properly on  $X$, and let $\calL$ be a big line sheaf.  For each $\varepsilon>0$,  there exists a positive integer $Q$  such that for any algebraically non-degenerate holomorphic map $f:\C\to X$,  the inequality 
\begin{align*}
    \sum_{j=1}^q\beta(\calL, D_j)T_{D_j,f}(r)-  (1+\varepsilon)T_{\calL,f}(r)\leq_{\operatorname{exc}} \sum_{j=1}^q \beta(\calL, D_j)N_f^{(Q)}(D_j,r) 
\end{align*}
    holds,  where $\leq_{\operatorname{exc}}$ means the inequality holds for all $r\in \R^+$ except a set of finite Lebesgue measure.
\end{theorem}
\begin{remark}
  When $X$ is a surface, the condition that $D_1,\hdots,D_q$ intersect properly on $X$ can be relaxed to a general position assumption, following the same strategy as in \cite[Main Theorem]{hussein-general}.
\end{remark} 
  
We recall the following theorem from \cite{GuoWang}, which is a modification of \cite[Theorem 3.2]{hussein-general}  by applying  the  general form of the second main theorem with a Wronskian term  proved by Vojta in \cite[Theorem 1]{vojta1997} and Ru in  \cite[Theorem 2.3]{ru1997general}. 
\begin{theorem}
\label{gsmt2}
    Let $X$ be a complex projective variety and let $D$ be a Cartier divisor on $X$, let $V$ be a nonzero linear subspace of $H^0(X,\mathcal{O}(D))$, let $s_1,\dots,s_q$ be nonzero elements of $V$ and for each $j=1,\dots,q$, let $D_j$ be the Cartier divisor $(s_j)$.

Let $\Phi=(\phi_1,\hdots,\phi_{d}):X\dashrightarrow \mathbb{P}^{d-1}$ be the rational map associated to the linear system $V$,  let $f:\mathbb{C}\to X$ be a holomorphic map with Zariski dense image and let $\Psi=(\psi_1,\hdots,\psi_{d}):\mathbb C\to   \mathbb{P}^{d-1}$ be a reduced form of $\Phi\circ f$, i.e.  $\Psi=\Phi\circ f$  and $\psi_1,\hdots,\psi_{d}$ are entire functions without common zeros.  Denote by $W(\Psi)$ the Wronskian of $\psi_1,\hdots,\psi_{d}$.  
    Then for any $\varepsilon>0$, 
    \[
      \int_0^{2\pi} \max_J \sum_{j\in J} \lambda_{D_j}\left(f(re^{i\theta})\right)\frac{d\theta}{2\pi}+N_{W(\Psi)}(0,r)+\operatorname{dim}V\cdot N_h(0,r)  \leq_{\operatorname{exc}} (\operatorname{dim}V+\varepsilon)T_{D,f}(r),
    \]
    where $J$ ranges over all subsets of $\{1,\dots,q\}$ such that the sections $(s_j)_{j\in J}$ are linearly independent, and $h$ is an an entire function such that $h\psi_i=\phi_i(f)$ for each $1\le i\le d$, i.e. $h$ is a gcd of $\phi_i(f)$, $1\le i\le d$. 
\end{theorem}

\begin{proof}
 Most of the proof is identical to the proof of \cite[Theorem 3.2]{hussein-general}, therefore we will omit some details and indicate the required adjustments.

 We may assume that $d>1$.
  Let $X'$ be the
closure of the graph of $\Phi$, and let $p\colon X'\to X$ and
$\phi\colon X'\to\mathbb P^{d-1}$ be the projection morphisms.  
Then there is an effective Cartier divisor $B$ on $X'$ such that,
for each nonzero $s\in V$, there is a hyperplane $H$ in $\mathbb P^{d-1}$
such that $p^{*}(s)-B=\phi^{*}H$.  Let $\tilde f:\mathbb C\to X'$ be the lifting of $f$.
Then $\phi\circ \tilde f=(\psi_1,\hdots,\psi_{d}):\mathbb C\to   \mathbb{P}^{d-1}$ is a reduced presentation of  $\Phi\circ f$.  Moreover,  let $h$ be an entire function such that $h\psi_i=\phi_i(f)$ for each $1\le i\le d$, i.e. $h$ is a gcd of $\phi_i(f)$, $1\le i\le d$.  Then
\begin{align}\label{countingcompare}
N_{\tilde f}(B,r)\ge N_h(0,r)+O(1).
\end{align}
  
For each $j=1,\dots,q$ let $H_j$ be the hyperplane in $\mathbb P^{d-1}$
for which 
\begin{align}\label{divisorrelation}
p^{*}D_j-B=p^{*}(s_j)-B=\phi^{*}H_j.
\end{align}
  Choose a Weil function $\lambda_B$
for $B$.  Then,  we have
\begin{align*}
\lambda_{p^{*}D_j}=\lambda_{\phi^{*}H_j}+\lambda_B+O(1).
\end{align*}
By the functoriality of Weil functions, 
\[
  \lambda_{p^{*}D_j}\left(\tilde f(z)\right)=\lambda_{D_j}\left(  f(z)\right) \qquad \text{ and } \qquad \lambda_{\phi^{*}H_j}\left(\tilde f(re^{i\theta})\right)=\lambda_{ H_j}\left(\phi(\tilde f(re^{i\theta}))\right).
\]

By the  general form of the second main theorem with a Wronskian term    in \cite[Theorem 1]{vojta1997} and \cite[Theorem 2.3]{ru1997general},
we have 
\begin{align*}
 \int_0^{2\pi} \max_J \sum_{j\in J} \lambda_{ H_j}\left(\phi(\tilde f)(re^{i\theta})\right)\frac{d\theta}{2\pi}+N_{W(\phi(\tilde f ))}(0,r) \leq_{\operatorname{exc}} (d+\varepsilon)T_{\phi(\tilde f )}(r).
\end{align*}
From \eqref{divisorrelation}, we have
\[
T_{\phi(\tilde f )}(r)=T_{D,f}-T_{B,\tilde f}.
\]
Since each set $J$ has at most $\dim V$ elements and $B$ is effective, it follows from \eqref{countingcompare} that
\[
(\# J)\left(m_{\tilde f}(B,r)+N_h(0,r)\right)\le (\# J)\left(m_{\tilde f}(B,r)+N_{\tilde f}(B,r)\right)\le d T_{B,\tilde f} +O(1).
\]
Hence,
\begin{align*} 
 &\int_0^{2\pi} \big(\max_J \sum_{j\in J}\lambda_{ H_j}\left(\phi(\tilde f)(re^{i\theta})\right)+\lambda_B\left(\tilde f (re^{i\theta})\right) \big)\frac{d\theta}{2\pi}+d N_h(0,r))+N_{W(\phi(\tilde f ))}(0,r)\\
 & \leq_{\operatorname{exc}} (d+\varepsilon)T_{\phi(\tilde f )}(r)+dT_{B,\tilde f} +O(1)\\
 & \leq_{\operatorname{exc}}(d+\varepsilon)T_{D,f}+O(1).
\end{align*}
\end{proof}

\begin{proof}[Proof of Theorem \ref{trungeneral}]
  We will follow the proof of  the analytic part of \cite[General Theorem]{RV19} closely, and only indicate the necessary modification.
Let $\epsilon>0$ be given.  We want to show that 
\begin{align}\label{mainineq1}
    \sum_{j=1}^q\beta(\calL, D_j)\left(T_{D_j,f}(r)-N_f^{(Q)}(D_j,r) \right) \leq_{\operatorname{exc}} (1+\varepsilon)T_{\calL,f}(r). 
\end{align}
Since the quantities  
 $\big(T_{D_j,f}(r)-N_f^{(Q)}(D_j,r)\big)/ T_{\calL,f}(r) $ 
are bounded when their respective
denominators are sufficiently large, it suffices to prove
(\ref{mainineq1}) 
with a slightly smaller $\epsilon>0$
and with $\beta(\calL,D_i)$ replaced by slightly smaller
$\beta_i\in\Q$ for all $i$.   It's also clear that we may assume that $\beta_i\ne 0$ for each $i$.

Choose positive integers $N$ and $b$ such that
\begin{equation}\label{aut_choices}
  \left( 1 + \frac nb \right) \max_{1\le i\le q}
      \frac{\beta_i Nh^0(X, \calL^N) }
        {\sum_{m\ge1} h^0(X, \calL^N(-mD_i))}
    < 1 + \epsilon\;.
\end{equation}

Let
$$\Sigma
  = \biggl\{\sigma\subseteq \{1,\dots,q\}
    \bigm| \bigcap_{j\in \sigma} \Supp D_j\ne\emptyset\biggr\}\;.$$
For $\sigma\in \Sigma$, let
$$\bigtriangleup_{\sigma}
= \left\{\mathbf a = (a_i)\in \prod_{i\in\sigma}\beta_i^{-1}\mathbb{N}
    \Bigm| \sum_{i\in\sigma} \beta_ia_i = b \right\}\;.$$
For $\mathbf a\in\bigtriangleup_{\sigma}$ as above, we construct a filtration of $H^0(X, \calL^N)$ as follows: 
For  $x\in  \R^+$, 
one defines  
the ideal $\mathcal I_{\mathbf a}(x)$ of ${\calO}_X$ by
\begin{equation*}
  \mathcal I_{\mathbf a}(x)
  = \sum_{{\bf b}} {\calO}_X\left(-\sum_{i\in \sigma} b_iD_{i}\right)
\end{equation*}
where the sum is taken for all ${\bf b}\in {\mathbb{N}}^{\#\sigma}$
with $\sum_{i\in \sigma} a_ib_i\ge bx$. Let
\[
  {\mathcal F}(\sigma; {\bf a})_x
  = H^0(X, \calL^N \otimes {\mathcal I_{\mathbf a}}(x)),
\]
which we regard as a subspace of $H^0(X,\calL^N)$.
We note that there are only finitely many ordered pairs
$(\sigma,\mathbf a)$ with $\sigma\in\Sigma$
and $\mathbf a\in\bigtriangleup_\sigma$.  Let $\mathcal B_{\sigma; {\bf a}}$ be a basis of
$H^0(X, \calL^N)$ adapted to the above filtration $\{{\mathcal F}(\sigma; {\bf a})_x\}_{x\in\R^+}$.

  For a basis ${\mathcal B}$ of $H^0(X, {\calL}^N)$, denote by $(\mathcal B)$  the sum of the divisors $(s)$
for all $s\in\mathcal B$.  
We now state the following main lemma in  \cite{RV19}.

\begin{lemma}[{\cite[Lemma 6.8]{RV19}}]\label{aut_main_lemma}
With the above notation, we have
\begin{equation}\label{aut_main_step}
  \bigvee_{\substack{\sigma\in\Sigma \\ \mathbf a\in\Delta_\sigma}}
      (\mathcal B_{\sigma;\mathbf a})
    \ge \frac b{b+n}\left(\min_{1\le i\le q}
      \frac{\sum_{m=1}^\infty h^0(X,\calL^N(-mD_i))}{\beta_i}\right)
      \sum_{i=1}^q \beta_iD_i\;.
\end{equation}
\end{lemma}
Here, the notation ``$\bigvee$" is referred to the least upper bound with respect the partial order on the (Cartier) divisor group of $X$  by the relation $D_1\le D_2$ if $D_2-D_1$ is effective.
 Write
$$\bigcup_{\sigma; {\bf a}}  \mathcal B_{\sigma; {\bf a}}    = \mathcal B_1\cup\cdots \cup \mathcal B_{T_1}=\{s_1, \dots, s_{T_2}\}.$$
 For each $i=1,\dots, T_1$, let $J_i\subseteq\{1,\dots,T_2\}$ be the subset
such that $\mathcal B_i = \{s_j:j\in J_i\}$.  Choose Weil functions
 $\lambda_{\mathcal B_i}$ ($i=1,\dots,T_1$),
and $\lambda_{s_j}$ ($j=1,\dots,T_2$) for the divisors $(\mathcal B_i)$,
and $(s_j)$, respectively.  Then, by (\ref{aut_main_step})
 for  $x\in X$, 
\begin{equation}\label{weilbase2}
  \begin{split}
  &\frac b{b+n} \left( \min_{1\leq i \leq q} \sum_{m\ge 1}
    \frac{h^0(\calL^N(-mD_i))}{\beta_i} \right)
    \sum_{i=1}^q  \beta_i \lambda_{D_i} (x) \\
  &\qquad\le \max_{1\le i\le T_1} \lambda_{\mathcal B_i}  (x)+ O(1)= \max_{1\le i\le T_1} \sum_{j\in J_i} \lambda_{s_j} (x) + O(1).
  \end{split}
\end{equation}
Let $M=h^0(X,\calL^N)$ an let the set $\{\phi_1,\dots,\phi_M\}$ be a basis of the vector space $H^0( X, \calL^N)$, and let $\Phi=[\phi_1,\dots,\phi_M]:X\to \PP^{M-1}(\C)$ be the corresponding (rational) map.  Let $f:\C\to X$ be an algebraically non-degenerate holomorphic map.  Let $h$ be a gcd of $\phi_1(f),\hdots,\phi_M(f)$, i.e. $h$ is an entire function such that  $\psi_1:=h^{-1} \phi_1(f),\hdots,\psi_M:=h^{-1} \phi_M(f)$ have no common zeros. Then  $\Psi=(\psi_1,\hdots,\psi_M)$ is a reduced form of $\Phi\circ f$.

By Theorem \ref{gsmt2},
\begin{equation}\label{schmidt_ineq}
  \int_0^{2\pi} \max_J \sum_{j\in J} \lambda_{s_j}(f(re^{i\theta}))\frac{d\theta}{2\pi}
    \leq_{\operatorname{exc}} \left(M+ \epsilon \right) T_{f,\mathscr L^N}(r) -N_{W(\Psi  )}(0,r)-MN_h(0,r),
\end{equation}
here the maximum is taken over all subsets $J$ of $\{1,\dots,T_2\}$
for which the sections $s_j$, $j\in J$, are linearly independent.
Combining (\ref{weilbase2}) and (\ref{schmidt_ineq}), we have
\begin{align*}
&\sum_{i=1}^q   \beta_i m_f(D_i,r)\\
  &\leq_{\operatorname{exc}} \left( 1 + \frac nb \right) \max_{1\le i\le q}
    \frac{\beta_i }
      {\sum_{m\ge1} h^0(\mathscr L^N(-mD_i))}\big( (M + \epsilon)T_{f,\mathscr L^N}(x) -N_{W(\Psi  )}(0,r)-MN_h(0,r)\big).  \end{align*}
 Using
(\ref{aut_choices}) and the fact that $T_{f,\mathscr L^N}(r)=NT_{f,\mathscr L }(r)$,
we have
\begin{equation*}
 \sum_{i=1}^q    \beta_i m_f(D_i,r)
  \leq \left( 1+2\epsilon \right)T_{f,\mathscr L}(x) -C(N_{W(\Psi  )}(0,r)+MN_h(0,r)),
 \end{equation*}
where $C= (1+\epsilon)(MN)^{-1}$.
Applying   \eqref{FMT} to the above inequality,    we have
    \begin{equation*}
            \sum_{j=1}^q\beta_jT_{D_j,f}(r)- (1+2\varepsilon)T_{  \mathscr L,f}(r)\leq_{\operatorname{exc}} \sum_{j=1}^q \beta_jN_f(D_j,r)-C(N_{W(\Psi  )}(0,r)+MN_h(0,r)).    
            \end{equation*}   
  To finish the proof, we will show that there exists a large integer $Q$ (to be determined later) such that 
  \begin{equation}
    \label{TN2}
        \begin{aligned}
       \sum_{j=1}^q \beta_jN_f(D_j,r)-C(N_{W(\Psi  )}(0,r)+MN_h(0,r))\le  \sum_{j=1}^q \beta_jN_f^{(Q)}(D_j,r).
        \end{aligned}
    \end{equation} 
 For $z_0\in \mathbb C$, let $\rho_j$ be a local defining function of $D_j$ around one of its open neighborhoods $U$.  
 To show (\ref{TN2}), it suffices to deduce the following inequality 
  \begin{equation*}
        \begin{aligned}
       \sum_{j=1}^q \beta_j v_{z_0}^+ ( \rho_j\circ f)-C \cdot (v_{z_0}^+ (W(\Psi))+M v_{z_0}  (h)) \le  \sum_{j=1}^q\beta_j \min\{Q,v_{z_0}^+  (\rho_j\circ f)\},
        \end{aligned}
    \end{equation*} 
 for each $z_0\in \mathbb C$.

The above inequality holds trivially if $v_{z_0}^+(\rho_j\circ f)\le  Q$ for each $1\le j\le q$.
Therefore, we assume that $v_{z_0}^+(\rho_j\circ f)\ge  Q$ for some $1\le j\le q$.
By Lemma \ref{aut_main_lemma},  there exists a set of basis $s_1,\hdots,s_M$ of $H^0( X,\mathscr L^N)$ with the following property:
\begin{align}\label{S2} 
\sum_{i=1}^M (s_i)  &\geq  \frac b{b+n}\left(\min_{1\le i\le q}\sum_{m=1}^\infty
      \frac{h^0(X,\mathscr L^N(-mD_i))}{\beta_i}\right)  \sum_{i=1}^{q}\beta_{i}D_{i}\cr
      &\ge \frac{MN}{1+\epsilon} \sum_{i=1}^{q}\beta_{i}D_{i}=\frac 1C\sum_{i=1}^{q}\beta_{i}D_{i},
\end{align}  
where the last inequality is due to   \eqref{aut_choices}.  
Since
\begin{align*}
v_{z_0}^+ (s_i|_U\circ f)\ge \sum_{j=1}^q\text{ord}_{D_j}s_i \cdot v_{z_0}^+ (\rho_j\circ f),
\end{align*}
 we can derive from (\ref{S2}) that
\begin{align}\label{S1} 
\sum_{i=1}^M v_{z_0}^+ (s_i|_U\circ f)  \geq  \frac1C\sum_{j=1}^{q}\beta_jv_{z_0}^+(\rho_j\circ f).
\end{align}      
On the other hand, since $\{\phi_1,\dots,\phi_M\}$ is a basis of the vector space $H^0( X,\mathscr L^N)$, each $s_j$ is a $ \mathbb C$-linear combination of the $\phi_i$'s, and hence each  $h^{-1}s_j(f)$ is a linear combination of the $\psi_i$'s.
From the basic properties of Wronskians, we have
 \begin{equation}
    \label{ordW}
        \begin{aligned}
     v_{z_0}^+(W(\Psi))\ge \sum_{j=1}^Mv_{z_0}^+ (s_j|_U\circ f) -M v_{z_0}(h)-\frac12M(M-1).        
      \end{aligned}
 \end{equation}  
Combining (\ref{S1}) and (\ref{ordW}) and the assumption that $v_{z_0}^+(\rho_j\circ f)\ge  Q$ for some $1\le j\le q$, we obtain 
\begin{equation*}
        \begin{aligned}
       \sum_{j=1}^q\beta_j v_{z_0}^+ (\rho_j\circ f)-C\cdot ( v_{z_0}^+(W(\Phi\circ f))+M v_{z_0}(h)) 
       &\le  \frac12CM(M-1)\cr 
       & \le \sum_{j=1}^q \beta_j\min\{Q,v_{z_0}^+ (\rho_j\circ f)\},
        \end{aligned}
    \end{equation*}
where $Q$ is chosen to be $\frac{CM(M-1)}{2\min_{1\le j\le q}\{\beta_j\}}$. (Note that we have assumed that $\beta_j\ne 0$ for each $j$.)
 This completes our proof.     
\end{proof}

\subsection{Function Fields}\label{subsec:ffRV}
 In this section we let $\kappa$ be an algebraically closed field of characteristic zero, let $\calC$ be a smooth projective curve over $\kappa$ of genus $g(\calC)$ and let $K=\kappa(\calC)$ be the function field of $\calC$.  
For each point ${\mathfrak p}\in \calC$, we may choose a uniformizer $t_{\mathfrak p}$ to define a
normalized order function $v_{\mathfrak p}:=\ord_{\mathfrak p}:K\to \Z\cup\{\infty\}$ at
${\mathfrak p}$.    For a non-zero element $f\in K$, the
\emph{height} $h(f)$ counts its number of poles with multiplicities, i.e.
\[
  h(f):=\displaystyle{\sum_{{\mathfrak p}\in\calC} \max \{ 0, - v_{\mathfrak p}(f)\}}.
\]
Let $f_0,...,f_m\in K$ not all zeros.  Then
${\bf f}=[f_0:\cdots:f_m]\in \PP^m(K)$ can be viewed as a morphism from $\calC$ to $\PP^m(\kappa)$. The height of this morphism (or of the corresponding point in $\PP^m(K)$) is  defined by
$$
h({\bf f})=h(f_0,...,f_m):=\sum_{{\mathfrak p}\in \calC} -\min\{ v_{\mathfrak p} (f_0),...,v_{\mathfrak p}(f_m)\}.
$$
Let $D$ be a Cartier divisor on a variety $X$ over $\kappa$.  Similarly to the number field case, the classical theory of heights (see e.g. \cite[Part B]{SH} or \cite[Chapter 2]{BG}) associates to every Cartier divisor $D$ on $X$ a height function $h_D: X(K)\to \R$ and a Weil function  (local height function)    $\lambda_{D,{\mathfrak p}}  :X(K)\setminus\Supp D\to \R$, ${\mathfrak p}\in \calC$, well-defined up to a bounded functions such that 
\begin{align}\label{FMTff}
\sum_{{\mathfrak p}\in \calC}\lambda_{D,{\mathfrak p}}(P)=h_D(P)+O(1)
\end{align} 
for all $P\in X(K)\setminus\Supp D$.
Let $S$ be a finite set  of points in $\calC$.   
 We denote by 
 $$
 m_{D,S}(P)=\sum_{{\mathfrak p}\in S}\lambda_{D,{\mathfrak p}}(P),\quad\text{and} \quad  N_{D,S}(P)=\sum_{{\mathfrak p}\notin S}\lambda_{D,{\mathfrak p}}(P)
 $$
 the proximity and counting functions as defined in \cite{Vojta87}.
 Let  $a_0X_0+a_1X_1+\cdots+a_mX_m$ be a linear form with $a_0,\hdots,a_m\in \kappa$ whose vanishing determine a hyperplane $H$  in $\PP^m$.  Then for all ${\mathfrak p}\in \calC$ and $P=[x_0:\cdots:x_m]\in \PP^m(K)$, 
the Weil function at   ${\mathfrak p} $  is given by  
$$
\lambda_{H,{\mathfrak p}}(P)=v_{\mathfrak p}(a_0x_0+a_1x_1+\cdots+a_mx_m)-\min\{ v_{\mathfrak p} (x_0),...,v_{\mathfrak p}(x_m)\}.
$$

We recall the following version of the second main theorem for function fields from  \cite{Wang}.
 \begin{theorem}[{\cite[Theorem 1]{Wang}}] \label{smtfun}
In the above setting, let $H_1,\hdots,H_q$ be hyperplanes in $\PP^m(K)$ defined by linear forms with coefficients in $\kappa$.
If $P=[x_0:\cdots:x_m]\in \PP^m(K)$ is linearly nondegenerate over $\kappa$, then
$$
\sum_{\mathfrak p\in S}\max_{j\in J} \lambda_{H_j,{\mathfrak p}}(P)\le (m+1)h(P)+\frac{m(m+1)}2(2g(\calC) - 2 +\# S),
$$
where the maximum is taken over all subset $J$ of $\{1,\hdots,q\}$ such that the linear forms defining $H_j$, $j\in J$, are linearly independent.
\end{theorem}

We will use Theorem \ref{smtfun} to obtain a function field analogue of Theorem \ref{trungeneral} with an explicit dependence on the Euler characteristic of the curve $\calC$.

 \begin{theorem} \label{thm:ffconst}
   Let $X\subset \PP^m$ be a projective variety over $\kappa$ of dimension $n$, let $D_1,\cdots,D_q$ be effective Cartier divisors intersecting properly on  $X$, and let $\calL$ be a big line sheaf.  
   Then for any $\epsilon>0$, there exist constants $c_1$ and $c_2$, independent of the curve $\calC$ and the set $S$, such that for any map $x=[x_0:\cdots:x_m] :\calC\to X$, where $x_i\in K$, outside the augmented base locus of $\calL$ we have 
 either
 \begin{align*}
   \sum_{i=1}^q \beta_{\mathcal L, D_i} m_{D_i,S}(x)\le (1+\epsilon) h_{\calL}(x)+c_1 \max\left\{ 1, 2g(\calC)-2+|S|\right\},
\end{align*}
or the image of $x$  is contained in  a hypersurface (over $\kappa$) in $\PP^m$ of degree at most $c_2$. 

\end{theorem}

\begin{proof}   
The proof is similar to the first part of the proof  of Theorem \ref{trungeneral}.  We will follow its argument and notation and only indicate the modification.
Let $\epsilon>0$ be given.  Since $\calL$ is a big line sheaf, there is a constant $c$ such that  $\sum_{i=1}^qh_{D_i}(x)\le c h_{\calL}(x)$ for all $x\in X(K)$ outside the augmented base locus $B$ of $\calL$. This follows from \cite[Proposition 10.11]{Vojta11}.
By \eqref{FMTff}, together with the fact that $m_{D_i,S}\le h_{D_i}+O(1)$, we can choose $\beta_i\in\Q$ for all $i$ such that 
\[
\sum_{i=1}^q (\beta_{\mathcal L, D_i}-\beta_i) m_{D_i,S}(x)\le \frac{\epsilon}{2}h_{\calL}(x)
\]
for all $x\in X\setminus B(K)$. Therefore, we can assume that
$\beta_{\mathcal L, D_i}=\beta_i\in\Q$ for all $i$.   It's also clear that we may assume that $\beta_i\ne 0$ for each $i$. From now on we will assume that the point $x \in X(K)$ does \emph{not} lie on $B$.

Choose positive integers $N$ and $b$  to satisfy \eqref{aut_choices}.  The same arguments in the proof  of Theorem \ref{trungeneral} give
 \begin{equation}\label{eq:weilbase_ff}
  \begin{split}
  &\frac b{b+n} \left( \min_{1\leq i \leq q} \sum_{m\ge 1}
    \frac{h^0(\calL^N(-mD_i))}{\beta_i} \right)
    \sum_{i=1}^q  \beta_i \lambda_{D_i,\mathfrak p} (x) \\
  &\qquad\le \max_{1\le i\le T_1} \lambda_{\mathcal B_i,\mathfrak p}  (x)+ O(1)= \max_{1\le i\le T_1} \sum_{j\in J_i} \lambda_{s_j,\mathfrak p} (x) + O(1).
  \end{split}
\end{equation}
Let $M=h^0(X,\calL^N)$, let the set $\{\phi_1,\dots,\phi_M\}$ be a basis of the vector space $H^0( X, \calL^N)$, and let 
\begin{align}\label{Phi}
\Phi=[\phi_1,\dots,\phi_M]:X \dashrightarrow \PP^{M-1}(\kappa)
\end{align}
 be the corresponding rational map.  By  Theorem \ref{smtfun},  either the map $\Phi\circ x$ is linearly  degenerate, i.e. $\phi_1(x),\dots,\phi_M(x)$ are linearly  dependent over $\kappa$, or 
\begin{equation}\label{eq:schmidt2_ff}
  \sum_{\mathfrak p\in S}\max_J \sum_{j\in J} \lambda_{s_j,\mathfrak p}(x) 
    \le M \, h_{\calL^N}(x)  +\frac{M(M-1)}2(2g-2+|S|),
\end{equation}
here the maximum is taken over all subsets $J$ of $\{1,\dots,T_2\}$
for which the sections $s_j$, $j\in J$, are linearly independent.
We first consider when $\phi_1,\dots,\phi_M$ are linearly  independent over $\kappa$.
Combining \eqref{eq:weilbase_ff} and \eqref{eq:schmidt2_ff} gives
$$\sum_{i=1}^q   \beta_i m_{D_i,S}(x)
  \leq \left( 1 + \frac nb \right) \max_{1\le i\le q}
    \frac{\beta_i }
      {\sum_{m\ge1} h^0(\calL^N(-mD_i))} \, M\, h_{\calL^N}(x)  +c_1'(2g-2+|S|) + O(1),$$
      where $c_1'=\frac{M(M-1)}2. $
 Using
\eqref{aut_choices} and the fact that $h_{\calL^N}(x)=Nh_{\calL}(x)$,
we have
\begin{equation*}
 \sum_{i=1}^q   \beta_i m_{D_i,S }(x)
 \leq \left( 1+\epsilon \right)h_{\calL}(x)+c_1'(2g-2+|S|) + O(1),
 \end{equation*}
 which implies the first case of the Theorem.

 To conclude we note that, if $\phi_1(x),\dots,\phi_M(x)$  are linearly  dependent over $\kappa$, there exist constants $a_1,\dots,a_M \in \kappa$, not all zero, such that $a_1\phi_1(x) + \dots + a_M \phi_M(x) = 0$. Let $H$ be the hyperplane in $\PP^{M-1}$ defined by $a_1z_1 + \dots + a_M z_M = 0$; by assumption $\Phi( x(\calC))$ is contained in $H$. On the other hand, since $\phi_1,\dots,\phi_M$ is a basis of $H^0(X,\calL^N)$, it follows that $\Phi(X)$ is not contained in $H$, hence $ \Phi(x(\calC))$ is contained in $\Phi(X) \cap H$ which is an hypersurface in $\Phi(X)$ whose degree is bounded independently of $\calC$ and $x$ as wanted.
\end{proof}
\begin{corollary}
  \label{cor:surf}
  In the previous setting, if $X$ has dimension 2, then for any $\epsilon>0$, there exist constants $c_1$ and $c_3$, independent of the curve $\calC$ and the set $S$, such that for any $K$-point $x=[x_0:\cdots:x_m] :\calC\to X$, with $x_i\in K$, we have that, either $\deg x(\calC) \leq c_3$, or 
 \begin{align*}
   \sum_{i=1}^q \beta(\mathcal L, D_i) m_{D_i,S}(x)\le (1+\epsilon) h_{\calL}(x)+c_1 \max\{ 1, 2g(\calC)-2+|S| \}.
\end{align*}
\end{corollary}

\begin{proof}
  Given Theorem \ref{thm:ffconst}, it is enough to observe that, since $X$ has dimension 2, the hypersurface of degree bounded by $c_2$ intersects $X$ in a union of finitely many curves whose degree is bounded independently of the point $x$. Moreover, the augmented base locus of $\calL$ has dimension at most 1, and its one dimensional locus is a union of finitely many curves, which are also independent of $x$. Therefore, it suffices to define $c_3$ to the be the maximum of the degrees of all these finitely many curves.
\end{proof}

\bibliography{references}{}

\begin{thebibliography}{{Bog}79}

\bibitem[ADT20]{advt}
Kenneth Ascher, Kristin DeVleming, and Amos Turchet.
\newblock {Hyperbolicity and Uniformity of Varieties of Log General type}.
\newblock {\em International Mathematics Research Notices}, 08 2020.
\newblock published online.

\bibitem[Aut11]{Aut2}
Pascal Autissier.
\newblock Sur la non-densit\'{e} des points entiers.
\newblock {\em Duke Math. J.}, 158(1):13--27, 2011.

\bibitem[BG06]{BG}
Enrico Bombieri and Walter Gubler.
\newblock {\em Heights in {D}iophantine geometry}, volume~4 of {\em New
  Mathematical Monographs}.
\newblock Cambridge University Press, Cambridge, 2006.

\bibitem[{Bog}79]{Bog}
Fedor~A. {Bogomolov}.
\newblock {Holomorphic tensors and vector bundles on projective varieties.}
\newblock {\em {Math. USSR, Izv.}}, 13:499--555, 1979.

\bibitem[BRT19]{LRT}
Federico~Lo Bianco, Erwan Rousseau, and Fr{\'e}d{\'e}ric Touzet.
\newblock Symmetries of transversely projective foliations.
\newblock {\em arXiv preprint arXiv:1901.05656}, 2019.

\bibitem[BT04]{BT}
Fedor Bogomolov and Yuri Tschinkel.
\newblock Special elliptic fibrations.
\newblock In {\em The {F}ano {C}onference}, pages 223--234. Univ. Torino,
  Turin, 2004.

\bibitem[Cam04]{Ca04}
Fr\'{e}d\'{e}ric Campana.
\newblock Orbifolds, special varieties and classification theory.
\newblock {\em Ann. Inst. Fourier (Grenoble)}, 54(3):499--630, 2004.

\bibitem[Cam11]{Ca11}
Fr\'{e}d\'{e}ric Campana.
\newblock Orbifoldes g\'{e}om\'{e}triques sp\'{e}ciales et classification
  bim\'{e}romorphe des vari\'{e}t\'{e}s k\"{a}hl\'{e}riennes compactes.
\newblock {\em J. Inst. Math. Jussieu}, 10(4):809--934, 2011.

\bibitem[Che04]{Chen}
Xi~Chen.
\newblock On algebraic hyperbolicity of log varieties.
\newblock {\em Commun. Contemp. Math.}, 6(4):513--559, 2004.

\bibitem[CP07]{CP}
Fr\'{e}d\'{e}ric Campana and Mihai P\u{a}un.
\newblock Vari\'{e}t\'{e}s faiblement sp\'{e}ciales \`a courbes enti\`eres
  d\'{e}g\'{e}n\'{e}r\'{e}es.
\newblock {\em Compos. Math.}, 143(1):95--111, 2007.

\bibitem[CT19]{CaTur}
Laura Capuano and Amos Turchet.
\newblock Lang--{V}ojta conjecture over function fields for surfaces dominating
  {$\mathbb{G}^2_m$}.
\newblock {\em arXiv preprint arXiv:1911.07562}, 2019.

\bibitem[CZ04]{CZAnnals}
Pietro Corvaja and Umberto Zannier.
\newblock On integral points on surfaces.
\newblock {\em Ann. of Math. (2)}, 160(2):705--726, 2004.

\bibitem[CZ08]{CZConic}
Pietro Corvaja and Umberto Zannier.
\newblock {Some cases of Vojta's Conjecture on integral points over function
  fields}.
\newblock {\em J. Algebraic Geometry}, 17:195--333, 2008.

\bibitem[CZ10]{CoZa10}
Pietro Corvaja and Umberto Zannier.
\newblock Integral points, divisibility between values of polynomials and
  entire curves on surfaces.
\newblock {\em Adv. Math.}, 225(2):1095--1118, 2010.

\bibitem[CZ13]{CZGm}
Pietro Corvaja and Umberto Zannier.
\newblock Algebraic hyperbolicity of ramified covers of {$\Bbb G^2_m$} (and
  integral points on affine subsets of {$\Bbb P_2$}).
\newblock {\em J. Differential Geom.}, 93(3):355--377, 2013.

\bibitem[Dem97]{De97}
Jean-Pierre Demailly.
\newblock Algebraic criteria for {K}obayashi hyperbolic projective varieties
  and jet differentials.
\newblock In {\em Algebraic geometry---{S}anta {C}ruz 1995}, volume~62 of {\em
  Proc. Sympos. Pure Math.}, pages 285--360. Amer. Math. Soc., Providence, RI,
  1997.

\bibitem[Fal94]{FaltingsLang}
Gerd Faltings.
\newblock The general case of {S}. {L}ang's conjecture.
\newblock In {\em Barsotti {S}ymposium in {A}lgebraic {G}eometry ({A}bano
  {T}erme, 1991)}, volume~15 of {\em Perspect. Math.}, pages 175--182. Academic
  Press, San Diego, CA, 1994.

\bibitem[FM94]{FrieMorg}
Robert Friedman and John~W. Morgan.
\newblock {\em Smooth four-manifolds and complex surfaces}, volume~27 of {\em
  Ergebnisse der Mathematik und ihrer Grenzgebiete (3)}.
\newblock Springer-Verlag, Berlin, 1994.

\bibitem[Gra65]{Gra65}
Hans Grauert.
\newblock Mordells {V}ermutung \"{u}ber rationale {P}unkte auf algebraischen
  {K}urven und {F}unktionenk\"{o}rper.
\newblock {\em Inst. Hautes \'{E}tudes Sci. Publ. Math.}, (25):131--149, 1965.

\bibitem[GW19]{GuoWang}
Ji~Guo and Julie Tzu-Yueh Wang.
\newblock Asymptotic gcd and divisible sequences for entire functions.
\newblock {\em Trans. Amer. Math. Soc.}, 371(9):6241--6256, 2019.

\bibitem[HR18]{hussein-general}
Saud Hussein and Min Ru.
\newblock A general defect relation and height inequality for divisors in
  subgeneral position.
\newblock {\em Asian J. Math.}, 22(3):477--491, 2018.

\bibitem[HS00]{SH}
Marc Hindry and Joseph~H. Silverman.
\newblock {\em Diophantine geometry}, volume 201 of {\em Graduate Texts in
  Mathematics}.
\newblock Springer-Verlag, New York, 2000.
\newblock An introduction.

\bibitem[HT00]{HT}
Joe Harris and Yuri Tschinkel.
\newblock Rational points on quartics.
\newblock {\em Duke Math. J.}, 104(3):477--500, 2000.

\bibitem[HT01]{HassettT}
Brendan Hassett and Yuri Tschinkel.
\newblock Density of integral points on algebraic varieties.
\newblock In {\em Rational points on algebraic varieties}, volume 199 of {\em
  Progr. Math.}, pages 169--197. Birkh\"{a}user, Basel, 2001.

\bibitem[JX20]{Jav19b}
Ariyan Javanpeykar and Junyi Xie.
\newblock {Finiteness Properties of Pseudo-Hyperbolic Varieties}.
\newblock {\em International Mathematics Research Notices}, 07 2020.
\newblock published online.

\bibitem[Kaw80]{Kawamata}
Yujiro Kawamata.
\newblock On {B}loch's conjecture.
\newblock {\em Invent. Math.}, 57(1):97--100, 1980.

\bibitem[Lan86]{Lan86}
Serge Lang.
\newblock Hyperbolic and {D}iophantine analysis.
\newblock {\em Bull. Amer. Math. Soc. (N.S.)}, 14(2):159--205, 1986.

\bibitem[Lan91]{Lang91}
Serge Lang.
\newblock {\em Number theory. {III} - Diophantine geometry}, volume~60 of {\em
  Encyclopaedia of Mathematical Sciences}.
\newblock Springer-Verlag, Berlin, 1991.

\bibitem[Laz04]{Lazarsfeld1}
Robert Lazarsfeld.
\newblock {\em Positivity in algebraic geometry. {I}}, volume~48 of {\em
  Ergebnisse der Mathematik und ihrer Grenzgebiete (3)}.
\newblock Springer-Verlag, Berlin, 2004.

\bibitem[Lev09]{Levin}
Aaron Levin.
\newblock Generalizations of {S}iegel's and {P}icard's theorems.
\newblock {\em Ann. of Math. (2)}, 170(2):609--655, 2009.

\bibitem[Man63]{Manin63}
Ju.~I. Manin.
\newblock Proof of an analogue of {M}ordell's conjecture for algebraic curves
  over function fields.
\newblock {\em Dokl. Akad. Nauk SSSR}, 152:1061--1063, 1963.

\bibitem[McQ98]{McQ98}
Michael McQuillan.
\newblock Diophantine approximations and foliations.
\newblock {\em Inst. Hautes \'Etudes Sci. Publ. Math.}, No. 87:121--174, 1998.

\bibitem[Nog82]{Nog81}
Junjiro Noguchi.
\newblock A higher-dimensional analogue of {M}ordell's conjecture over function
  fields.
\newblock {\em Math. Ann.}, 258(2):207--212, 1981/82.

\bibitem[Rou10]{Rou}
Erwan Rousseau.
\newblock Hyperbolicity of geometric orbifolds.
\newblock {\em Trans. Amer. Math. Soc.}, 362(7):3799--3826, 2010.

\bibitem[Ru97]{ru1997general}
Min Ru.
\newblock On a general form of the second main theorem.
\newblock {\em Trans. Amer. Math. Soc.}, 349(12):5093--5105, 1997.

\bibitem[RV20]{RV19}
Min Ru and Paul Vojta.
\newblock A birational nevanlinna constant and its consequences.
\newblock {\em American Journal of Mathematics}, 142(3):957--991, 2020.

\bibitem[Tur17]{Tur}
Amos Turchet.
\newblock Fibered threefolds and {L}ang-{V}ojta's conjecture over function
  fields.
\newblock {\em Trans. Amer. Math. Soc.}, 369(12):8537--8558, 2017.

\bibitem[vBJK19]{Jav19}
Raymond van Bommel, Ariyan Javanpeykar, and Ljudmila Kamenova.
\newblock Boundedness in families with applications to arithmetic
  hyperbolicity.
\newblock {\em arXiv preprint arXiv:1907.11225}, 2019.

\bibitem[Voj87]{Vojta87}
Paul Vojta.
\newblock {\em {Diophantine Approximations and Value Distribution Theory}},
  volume 1239 of {\em Lecture Notes in Mathematics}.
\newblock Springer Berlin Heidelberg, 1987.

\bibitem[Voj97]{vojta1997}
Paul Vojta.
\newblock On {C}artan's theorem and {C}artan's conjecture.
\newblock {\em Amer. J. Math.}, 119(1):1--17, 1997.

\bibitem[Voj11]{Vojta11}
Paul Vojta.
\newblock Diophantine approximation and {N}evanlinna theory.
\newblock In {\em Arithmetic geometry}, volume 2009 of {\em Lecture Notes in
  Math.}, pages 111--224. Springer, Berlin, 2011.

\bibitem[Wan04]{Wang}
Julie Tzu-Yueh Wang.
\newblock An effective {S}chmidt's subspace theorem over function fields.
\newblock {\em Math. Z.}, 246(4):811--844, 2004.

\bibitem[Yam15]{Yamanoi}
Katsutoshi Yamanoi.
\newblock Holomorphic curves in algebraic varieties of maximal {A}lbanese
  dimension.
\newblock {\em Internat. J. Math.}, 26(6):1541006, 45, 2015.

\end{thebibliography}
\bibliographystyle{alpha}

\end{document}